\documentclass[12pt,reqno ]{amsart}

\usepackage{ulem}
\usepackage[makeroom]{cancel}
\usepackage{amsmath}
\usepackage{amsrefs}
\usepackage{amsthm}
\usepackage{amssymb}
\usepackage{amsfonts}
\usepackage{mathrsfs}
\usepackage{comment}
\usepackage{hyperref}
\usepackage{enumerate,color}

\numberwithin{equation}{section}

\newtheorem{theorem}{Theorem}[section]
\newtheorem{lemma}[theorem]{Lemma}
\newtheorem{corollary}[theorem]{Corollary}

\newtheorem{proposition}[theorem]{Proposition}
\newtheorem{conjecture}[theorem]{Conjecture}

\theoremstyle{definition}
\newtheorem*{remark}{Remark}

\def\XXint#1#2#3{{\setbox0=\hbox{$#1{#2#3}{\int}$}
\vcenter{\hbox{$#2#3$}}\kern-.5\wd0}}

\newcommand\numberthis{\addtocounter{equation}{1}\tag{\theequation}}

\renewcommand{\div}{\ensuremath{\text{div} }}

\newcommand{\C}{\ensuremath{\mathbb{C}}}
\newcommand{\Cn}{\ensuremath{\mathbb{C}^n}}

\newcommand{\R}{\ensuremath{\mathbb{R}}}
\newcommand{\Z}{\ensuremath{\mathbb{Z}}}
\newcommand{\Rd}{\ensuremath{\mathbb{R}^d}}
\newcommand{\Rn}{\ensuremath{\mathbb{R}^n}}
\newcommand{\D}{\ensuremath{\mathscr{D}}}

\newcommand{\MC}[1]{\ensuremath{\mathscr{#1}}}
\newcommand{\J}{\ensuremath{\mathscr{J}}}
\newcommand{\F}{\ensuremath{\mathscr{F}}}
\newcommand{\G}{\ensuremath{\mathscr{G}}}

\newcommand{\BMOW}{\ensuremath{{\text{BMO}}_W ^p}}

\newcommand{\Mn}{\ensuremath{\mathcal{M}_{n \times n }}(\mathbb{C})}
\newcommand{\Md}{\ensuremath{\mathcal{M}_{d \times d}}(\mathbb{C})}
\newcommand{\Mnd}{\ensuremath{\mathcal{M}_{n \times d}}(\mathbb{C})}
\newcommand{\MndR}{\ensuremath{\mathcal{M}_{n \times d}}(\mathbb{R})}
\newcommand{\MnR}{\ensuremath{\mathcal{M}_{n \times n }}(\mathbb{R})}
\newcommand{\MdR}{\ensuremath{\mathcal{M}_{d \times d}}(\mathbb{R})}

\newcommand{\V}[1]{\ensuremath{\vec{#1}}}
\newcommand{\inrd}{\ensuremath{\int_{\Rd}}}
\newcommand{\tr}{\ensuremath{\text{tr}}}
\newcommand{\ip}[2]{\ensuremath{\left\langle#1,#2\right\rangle}}
\newcommand{\iptr}[2]{\ensuremath{\left\langle#1,#2\right\rangle}_{\tr}}
\newcommand{\W}[1]{\ensuremath{\widetilde{#1}}}

\newcommand{\Atwo}[1]{\ensuremath[#1]_{\text{A}_2}}
\newcommand{\Apq}[1]{\ensuremath [#1]_{\text{A}_{p,q}}}
\newcommand{\Btwo}[1]{\ensuremath\|#1\|_{\text{B}_2}}
\newcommand{\pdisk}{\ensuremath{\partial \mathbb{D}}}
\newcommand{\disk}{\ensuremath{\mathbb{D}}}
\newcommand{\ind}{\ensuremath{\int_{\disk}}}

\newcommand{\A}{\ensuremath{{\mathcal A}}}
\newcommand{\B}{\ensuremath{{\mathcal B}}}
\newcommand{\al}{\alpha}
\newcommand{\ra}{\rightarrow}
\newcommand{\N}{\ensuremath{{\mathcal N}}}
\newcommand{\Ap}[1]{\ensuremath[#1]_{\text{A}_p}}
\newcommand{\Apprime}[1]{\ensuremath [#1]_{\text{A}_{p'}}}
\newcommand{\pr}[1]{\left(#1\right)}
\DeclareMathOperator{\Div}{Div}
\DeclareMathOperator{\dv}{div}
\allowdisplaybreaks

\title[]{Matrix weighted Poincar\'e inequalities and applications to degenerate elliptic systems}

\author{Joshua Isralowitz}
\address{Department of Mathematics and Statistics\\
SUNY Albany, 1400 Washington Ave., Albany, NY, 12222.}\email{jisralowitz@albany.edu}

\author{Kabe Moen}
\address{Department of Mathematics\\
University of Alabama, Box 870350, 345 Gordon Palmer Hall.}\email{kabe.moen@ua.edu}

\begin{document}

\subjclass[2010]{Primary 42B20}

\keywords{Matrix $A_p$ weights, Poincar\'e inequalities, elliptic PDE, fractional operators}

\thanks{The second is partially supported by the NSF under grant DMS 1201504}

\maketitle

\begin{abstract}
We prove Poincar\'{e} and Sobolev inequalities in matrix A${}_p$ weighted spaces.  We then use these Poincar\'{e} inequalities to prove existence and regularity results for degenerate systems of elliptic equations whose degeneracy is governed by a matrix A${}_p$ weight.  Such results parallel earlier results by Fabes, Kenig, and Serapioni for a single degenerate equation governed by a scalar A${}_p$ weight.  In addition, we prove Cacciopoli and reverse Meyers H\"older inequalities for weak solutions of the degenerate systems.  Moreover, we show that the Riesz potential and fractional maximal operators are bounded on matrix weighted $L^p$ spaces and go on to develop an entire matrix A${}_{p, q}$ theory. \end{abstract}

\section{Introduction}

The classic Poincar\'e inequality
$$\left(\,\frac{1}{|Q|}\int_Q |u(x)-u_Q|^q\,dx\right)^{1/q}\lesssim |Q|^{\frac1d}\left(\,\frac{1}{|Q|}\int_Q |\nabla u(x)|^p\,dx\right)^{1/p}$$
holds for all cubes $Q$ in $\R^d$ when $u$ is sufficiently smooth, $1\leq p<d$, and $q=\frac{dp}{d-p}$.  Such inequalities are vital to the regularity theory of weak solutions to elliptic PDEs.  Fabes, Kenig, and Serapioni \cite{FKS} studied the degenerate elliptic PDE \begin{equation}\label{elliptic} \text{div}\, (A(x)\nabla u(x))=\sum_{\al,\beta=1}^d \partial_\al(A_\al^\beta(x)\partial_\beta u(x))= - (\div \V{f})(x) \end{equation}
where $A$ is a positive definite matrix that satisfies
$$w(x)|\xi|^2\simeq \langle A(x)\xi,\xi\rangle, \quad \xi\in \R^d$$
for some $w\in A_2$ and $|\V{f}| \in L^2(\Omega, w^{-1})$ for some domain $\Omega \subseteq  \Rd$.  They proved (in conjunction with a result by Modica from \cite{M}) that weighted Poincar\'e inequalities of the form
\begin{multline*}\left(\,\frac{1}{w(Q)}\int_Q |u(x)-u_Q|^{p + \epsilon} w(x)\,dx\right)^{\frac{1}{{p + \epsilon}}}\\ \lesssim |Q|^{\frac1d}\left(\,\frac{1}{w(Q)}\int_Q |\nabla u(x)|^{p - \epsilon} w(x)\,dx\right)^{\frac{1}{{p - \epsilon}}}\end{multline*}
hold for some $\epsilon > 0$ when $w\in A_p$ and used these inequalities to prove that weak solutions to \eqref{elliptic} (under further assumptions on $\vec{f}$) are H\"older continuous.

In this paper we will more generally consider systems of degenerate elliptic equations of the form
\begin{equation}\label{linearsys}\sum_{j = 1}^n \sum_{\al,\beta = 1}^d  \partial_\al(A_{ij}^{\al\beta}(x) \partial_\beta u_j(x))= -(\div{F})_i (x), \qquad i=1,\ldots,n\end{equation}
for $n \in \mathbb{N}$ not necessarily equal to $d$, where $A_{ij}^{\al \beta} \in \C$ and
\begin{equation}\label{MIEllip1}\sum_{i, j = 1}^n \sum_{\al,\beta = 1}^d A_{ij}^{\al\beta}(x)  \eta_\beta^j \overline{\eta_{\al}^i} \gtrsim \|W(x)^{\frac12}\eta\|^2, \quad \eta \in \Mnd\end{equation} and

\begin{equation} \label{MIEllip2}  |\sum_{i, j = 1}^n \sum_{\al,\beta = 1}^d   A_{ij}^{\alpha \beta} (x) \nu_\beta ^j    \overline{\eta_{\alpha}^i}| \lesssim \|W^{\frac12}(x)\eta\| \|W^{\frac12}(x)\nu\|, \quad \eta, \nu \in \Mnd\ \  \end{equation}
for a matrix weight $W$ (i.e. an a.e. positive definite  $\Mn$  valued function with locally integrable entries) and $F \in L^2(\Omega, W^{-1})$ (which will be defined momentarily).  To the best of our knowledge, it seems that systems of elliptic equations whose degeneracies are governed by matrix weights have never been considered before.

Given a matrix weight $W$ and an exponent $p > 0$ we define $L^p (\Omega,W)$ to be the collection of all $\Cn$ valued functions $\vec{f}$ such that
\begin{equation*} \|\V{f}\|_{L^p (\Omega,W)}^p = \int_\Omega |W^{\frac1p}(x) \V{f}(x) |^p \, dx < \infty.
\end{equation*}
We will also sometimes let $L^p (\Omega,W)$ denote the space of all $\Mnd$ valued functions $F$ whose norm above is finite.  When $\Omega=\R^d$ we will write $L^p(W)$.

A natural solution space for weak solutions of \eqref{linearsys} is the matrix weighted Sobolev space ${\text H}^{1,p}(\Omega,W)$.  Define the norm by
$$\|\vec{f}\|_{\text{H}^{1,p}(\Omega,W)}=\left(\int_\Omega |W^\frac1p(x)\vec{f}(x)|^p\,dx\right)^{\frac1p}+\left(\int_\Omega \|W^{\frac1p}(x)D\vec{f}(x)\|^p\,dx\right)^{\frac1p}$$
where $\|\cdot\|$ is any  matrix norm.  The space H$^{1,p}(\Omega,W)$ is defined as the completion of $C^\infty (\Omega)$ with respect to the norm $\|\cdot\|_{\text{H}^{1,p}(\Omega,W)}$ (we refer the reader to section 5 for more on the space $\text{H}^{1,p}$).

A matrix weight $W$ belongs to A$_p$ if
\begin{equation*}[W]_{A_{p}}=\sup_{Q } \,  \frac{1}{|Q|} \int_Q \left( \frac{1}{|Q|}  \int_Q \|W^\frac{1}{p} (x) W^{-\frac{1}{p}}(y) \|^{p'} \, dy \right)^\frac{p}{p'} \, dx  < \infty. \end{equation*} When $p = 2$ {\color{red} we have} \begin{equation*}[W]_{A_{2}} \simeq \sup_{Q } \,  \tr  \, \left( \frac{1}{|Q|}  \int_Q W (x) \, dx \right) \left(\frac{1}{|Q|} \int_Q  W^{-1} (x) \, dx\right)  \end{equation*} which says that the matrix A${}_2$ condition is especially easy to verify.

Treil-Volberg \cite{TV} showed that the Hilbert transform, defined component-wise, is bounded on $L^2(W)$ if and only if the matrix weight $W$ belongs to A$_2$.  Nazarov-Treil and Volberg \cite{NT,V} when $d = 1$ and the first author \cite{I} when $d > 1$ proved dyadic upper and lower matrix weighted Littlewood-Paley $L^p$ bounds when $W$ is a matrix A${}_p$ weight. Furthermore,  Goldberg \cite{G} characterized the boundedness of singular integral operators and the Hardy-Littlewood maximal operator by the matrix A$_p$ condition.

We are now ready to state our main results.  We begin with Sobolev and Poincar\'e inequalities in the matrix weighted case. {\color{red} As is done implicitly in \cite{FKS},  we also restrict ourselves to the case $d \geq 2$ here and throughout the rest of this paper.}

\begin{theorem} \label{MainThmSob} If $1<p<\infty$ and $W$ is a matrix A${}_p$ weight, then there exists $\epsilon \approx \Ap{W} ^{- \max\{1, \frac{p'}{p}\}}$ and $C \approx \Ap{W} ^{\max\{ 1 + \frac{2p'}{p}, 2 + \frac{p'}{p}\}}$ where
\begin{multline*}
\left(\frac{1}{|Q|} \int_Q |W^\frac{1}{p} (x) \V{f}(x) |^{p+\epsilon } \, dx \right)^{\frac{1}{p+\epsilon}} \\ \leq
C |Q|^\frac{1}{d} \left(\frac{1}{|Q|} \int_Q \|W^\frac{1}{p} (x) D \V{f}(x)\|^{p-\epsilon } \, dx \right)^\frac{1}{p-\epsilon}  \end{multline*} for each cube $Q$ and $\vec{f}\in C_0^1(Q)$. \end{theorem}

\begin{theorem} \label{MainThmPoin} If $1<p<\infty$ and $W$ is a matrix A${}_p$ weight, then there exists $\epsilon \approx \Ap{W} ^{- \max\{1, \frac{p'}{p}\}}$ and $C \approx \Ap{W} ^{\max\{ 1 + \frac{2p'}{p}, 2 + \frac{p'}{p}\}}$ where
\begin{multline*}
\left(\frac{1}{|Q|} \int_Q |W^\frac{1}{p} (x) (\V{f}(x) - \V{f}_Q) \,  |^{p+\epsilon}  \, dx \right)^{\frac{1}{p+\epsilon}} \\ \leq
C |Q|^\frac{1}{d} \left(\frac{1}{|Q|} \int_Q \|W^\frac{1}{p} (x) D \V{f}(x)\|^{p-\epsilon } \, dx \right)^{\frac{1}{p-\epsilon}}  \end{multline*} for each cube $Q$ and $\V{f} \in C^1 (Q)$.\end{theorem}

As will be apparent from the proof, we can in fact replace the cube $Q$ in Theorems \ref{MainThmSob} and \ref{MainThmPoin} with an open ball $B$ so long as we have that $\vec{f} \in C_0 ^1(B)$ (respectively $\vec{f} \in C^1(B)$). Further, despite the unpleasant appearance of $\epsilon$ and $C$ in general, note that when $p = 2$ (which is most relevant for applications) we have $\epsilon \approx [W]_{\text{A}_2} ^{-1}$ and $C \approx [W]_{\text{A}_2} ^ 3$.  The upper dependence on $\epsilon$ is sharp even in the scalar case by considering the power weights of the form $w_s(x)=|x|^{n-s}$ as $s\ra 0$.  Note that the discrepancy between this bound and the one for fractional integral operators is  due to the more complicated stopping time argument needed to prove Theorems \ref{MainThmSob} and \ref{MainThmPoin} (see Lemma \ref{StoppingTimeLemma}.)   Moreover, it is very likely that the A${}_p$ dependence on $C$ above can be improved when $\epsilon = 0$ (via a different proof), though we will not explore this possibility in this paper.

Before we continue we will make an important comment on the results in \cite{FKS,M} compared to Theorems \ref{MainThmSob} and \ref{MainThmPoin}.  Note that Theorems \ref{MainThmSob} and \ref{MainThmPoin} contain the term $|Q|^{-1}$ as compared to the weighted measure of a cube (or really a ball) in \cite{FKS,M}.  Most likely this is due to the fact that the ``matrix weighted measure" of a level set makes no sense. In particular, this fact forces us to use a much different matrix weighted maximal function to prove Theorems \ref{MainThmSob} and \ref{MainThmPoin} as compared to the maximal function used in \cite{FKS,M}. It is not even clear what an appropriate replacement for the weighted measure of a cube/ball would be in the matrix case.  We also note that the Poincar\'e and Sobolev inequalities contained in \cite{FKS} show gains on the left of the form $1\leq q\leq \frac{n}{n-1}p+\delta$ for some $\delta>0$.  However, our Poincar\'e inequalities have gains on both the left and the right and it is for this reason (among those mentioned) that we do not obtain the same sharp exponents that are contained in \cite{FKS}.  Furthermore, because of this fact, one would be hard pressed to even attempt to define the ``matrix weighted weak type space" $L^{p, \infty} (W)$ when $W$ is a matrix weight.  {\color{red} Despite this, it would be interesting to know whether the weak type/truncation method used in \cite{LN,LMPT} can be used to prove results related to the ones proved in this paper.}

It is well known that Poincar\'e inequalities follow from bounds on the fractional integral operators
 $$I_\al f(x)=\int_{\R^d}\frac{f(y)}{|x-y|^{d-\al}}\,dy, \qquad 0<\al<d$$
 and their corresponding fractional maximal operators
 $$M_\al f(x)=\sup_{Q\ni x}\frac{1}{|Q|^{1-\frac{\al}{d}}}\int_Q|f(y)|\,dy, \qquad 0\leq \al<d.$$
Such operators play a crucial role in the theory of the smoothness of functions.  The fractional integral operator acts as an anti-derivative and hence its boundedness implies the Sobolev embedding theorems. While our Sobolev and Poincar\'e inequalities will not necessarily follow from matrix weighted bounds for fractional integral operators (and in fact the proof of the former in the local setting will be quite a bit more involved then then proof of the latter), we will nevertheless be interested in proving such bounds for their own sake.  Note that this is in contrast with the scalar situation where the above mentioned results in \cite{FKS,M} rely heavily on the classical weighted norm inequalities for fractional integral operators from \cite{MW1}.

First let us recall the results in the scalar case.  Muckenhoupt and Wheeden \cite{MW} characterized the weights $w$ for which
$M_\al$ and $I_\al$ are bounded on weighted Lebesgue spaces.  In particular, they showed that if $1<p< d / \alpha $ and $q$ is defined by $\frac1q=\frac1p-\frac{\al}{d}$, then $I_\al$ and $M_\al$ are bounded from $L^p(w^{\frac{p}{q}})$ to $L^q(w)$ if and only if $w\in A_{p,q}$:
$$[w]_{A_{p,q}}=\sup_{Q}\left(\frac{1}{|Q|}  \int_Q w(x)\,dx\right)\left(\frac{1}{|Q|}  \int_Q w(x)^{-\frac{p'}{q}}\,dx\right)^{\frac{q}{p'}}<\infty.$$
  Lacey et. al. \cite{LMPT}, found the sharp uppers bounds on the operator norms in terms of the constant $[w]_{A_{p,q}}$ showing that
\begin{equation}\label{Malsharp} \|M_\al\|_{L^p(w^{\frac{p}{q}})\ra L^q(w)}\lesssim [w]_{A_{p,q}}^{(1-\frac{\al}{d})\frac{p'}{q}}\end{equation}
and
 \begin{equation}\label{Ialsharp}\|I_\al\|_{L^p(w^{\frac{p}{q}})\ra L^q(w)}\lesssim [w]_{A_{p,q}}^{(1-\frac{\al}{d})\max(1,\frac{p'}{q})}.\end{equation}

We will study the matrix weighted case of these results.  Given a matrix weight $W$ and a pair of exponents $p$ and $q$ we define the matrix A$_{p,q}$ constant as follows
%\label{Apqcond}

\begin{equation*}[W]_{A_{p,q}}=\sup_{Q } \, \frac{1}{|Q|}  \int_Q \left( \frac{1}{|Q|}  \int_Q \|W^\frac{1}{q} (x) W^{-\frac{1}{q}}(y) \|^{p'} \, dy \right)^\frac{q}{p'} \, dx, \end{equation*}
where the supremum is over all cubes contained in $\R^d$.   A matrix weight $W$ belongs to A$_{p,q}$ if $[W]_{A_{p,q}}<\infty$.  We define the matrix weighted fractional maximal function as follows
\begin{equation*} M_{W, \alpha}  \V{f} (x) = \sup_{Q \ni x } \frac{1}{|Q|^{1 - \frac{\alpha}{d}}} \int_Q |W^{\frac1q}(x) W^{-\frac{1}{q}} (y) \V{f} (y)| \, dy \end{equation*}
where the supremum is over all cubes that contain $x$.  We will be concerned with $L^p\ra L^q$ bounds for $M_{W, \al}$.  Our first result is the following.
\begin{theorem} \label{thm:max}Suppose $0 \leq  \al<d$, $1<p<\frac{d}{\al}$ and $q$ is defined by $\frac1q=\frac1p-\frac{\al}{d}$.  If $W\in \text{A}_{p,q}$ then
\begin{equation}\label{MatrixMalsharp} \|M_{W,\al}\|_{L^p\ra L^q}\lesssim [W]_{\text{A}_{p,q}}^{\frac{p'}{q}(1-\frac{\al}{d})}\end{equation}
and this bound is sharp.
\end{theorem}Inequality \eqref{MatrixMalsharp} is the matrix valued version of \eqref{Malsharp}.  In fact, the sharpness of \eqref{MatrixMalsharp} follows from the scalar case because a better bound for the matrix case would imply a better bound for the scalar case.  We remark that the proof is a modification of the arguments found in \cite{Ae,G}.

For the fractional integral operator we have the following result.

\begin{theorem} \label{thm:frac}Suppose $0 <  \al<d$, $1<p< d/\al$ and $q$ is defined by $\frac1q=\frac1p-\frac{\al}{d}$.  If $W\in \text{A}_{p,q}$ then $I_\al:L^p(W^{\frac{p}{q}})\ra L^q(W)$ and
\begin{equation}\label{Matrixfracint}\|I_\al\|_{L^p(W^{\frac{p}{q}})\ra L^q(W)}\lesssim  [W]_{\text{A}_{p,q}}^{(1 - \frac{\alpha}{d}) \frac{p'}{q} + \frac{1}{q'}} \end{equation}
\end{theorem}

\noindent Let us make a remark about the bound above.  Formally when $\alpha = 0$ and $p = q = 2$ we get a bound of $[W]_{\text{A}_{2}}^{\frac32}$ which (thanks to Lemma \ref{DyadicLem} in Section \ref{fractionalbds}) is  closely related to  matrix weighted $L^2$ bounds for Calder\'on-Zygmund operators, the best of which at the moment is in fact $[W]_{\text{A}_{2}}^{\frac32}$ (see \cite{IKP,BW} for a proof of this fact for sparse operators and \cite{CDO,NPTV} where such bounds are used to get the matrix weighted $L^2$ bound of $[W]_{\text{A}_{2}}^{\frac32}$ for general CZOs).  Thus, while the bound in \eqref{Matrixfracint} is most likely not sharp, any improvement to (or ideas used to improve) \eqref{Matrixfracint} will most likely lead to improvements to matrix weighted bounds for CZOs, which is known to be very difficult.

Using our Poincar\'e inequalities we are able to prove regularity results for weak solutions to \eqref{linearsys}.  We begin with the following reverse H\"older inequality.  In the uniformly elliptic case, this classical result is due to Meyers.

\begin{theorem} \label{revmeyer} Let $W$ be a matrix A${}_2$ weight, let $\Omega$ be a domain in $\R^d$, and let $W^{-\frac12} F \in L^r(\Omega)$ for some $r > 2$.     If $A = A_{ij}^{\alpha \beta}$ satisfies \eqref{MIEllip1} and \eqref{MIEllip2}, and if $\vec{u} \in H^{1,2}(\Omega, W)$     is a weak solution to \eqref{linearsys}, then there exists $q > 2$ such that given $B_{2r} \subset \Omega$ we have \begin{align*}   \left( \frac{1}{|B_{r/2}|} \int_{B_{r/2}} \|W^\frac{1}{2} (x) D\V{u}  (x)\|^{q} \, dx \right)^\frac{1}{q}  & \lesssim \left( \frac{1}{|B_{r}|} \int_{B_{r}} \|W^\frac{1}{2}  (x) D\V{u} (x) \|^2 \, dx \right)^\frac{1}{2}  \\ & + \left(\frac{1}{|B_r|} \int_{B_r} \|W^{-\frac{1}{2}} (x) F(x)\|^{q} \, dx \right)^\frac{1}{q}. \end{align*} \end{theorem}

     In Section $6$ we will also use some of the ideas in the recent paper \cite{FBT} to extend this Meyers reverse H\"{o}lder inequality to solutions of nonhomogenous degenerate $p-$Laplacian systems with a matrix A${}_p$ degeneracy. Namely we will prove the following.

\begin{theorem} \label{FBTThm} Let $p > 2$.  Suppose that $G : \Omega \rightarrow \Md$ and $W$ is a matrix weight A${}_p$ weight.  Let $G$ satisfy

\vspace{3mm}

(i') $ \ip{\eta G(x)}{\eta}_{\tr}   \gtrsim \|W^{1/p}(x)\eta\|^2,  \qquad \eta \in {\color{red}  \Mnd}$

(ii') $|\ip{\eta G(x)}{\nu}_{\tr} | \lesssim  \|W^{1/p}(x)\eta\|\|W^{1/p} (x)\nu\|, \qquad \eta,\nu\in \Mnd$

\noindent and assume there exists $r > p'$ where $W^{-\frac{1}{p}} F \in L^{r}(\Omega)$.  If $\vec{u} \in H^{1,p}(\Omega, W)$ is a weak solution to  \begin{equation} \label{PLap}  \Div \left[ \ip{D\V{u} G}{D\V{u}}_{\tr} ^{\frac{p-2}{2}} D\vec{u} G \right] = - \Div F \end{equation} then there exists $q > p$ such that given $ B_{2r} \subset \Omega$ we have \begin{align*} \left(\frac{1}{|B_{r/2} |} \int_{B_{r/2}} \|W^\frac{1}{p} (x) D\V{u} (x)\|^{q} \, dx \right)^\frac{1}{q} & \lesssim \left(\frac{1}{|B_{r}|} \int_{B_{r}} \|W^\frac{1}{p} (x) D\V{u} (x)\|^p \, dx \right)^{\frac{1}{p}} \\ & + \left(\frac{1}{|B_r|} \int_{B_r} \|W^{-\frac{1}{p}}(x) F(x)\|^{\frac{q p'}{p}} \, dx \right)^\frac{1}{q}. \end{align*} \end{theorem}

Furthermore, we will prove the existence of weak solutions to \eqref{PLap} for all $p > 1$  in the special case when $ F = 0, \vec{u} : \Omega \rightarrow \Rn, G : \Omega \rightarrow {\color{red}\MdR},$ and $W : \Rd  \rightarrow\MnR$ (see Theorem \ref{nonlinexist}).  Intriguingly, note that (i') and (ii') are easily seen to be immediately satisfied when $n  = d$ and $G$ is itself an $d \times d$ matrix weight on $\Rd$ and $W = G^\frac{p}{2}$. Thus, when $n = d$, this gives us the existence of weak solutions and Theorem \ref{FBTThm} (when $p > 2$) for \eqref{PLap} when $G ^\frac{p}{2}$ is itself a matrix A${}_p$ weight, which are two results that we believe are of independent interest.

Finally, we end with the last of our main result: a local regularity theorem for weak solutions in dimension two.
\begin{theorem} \label{thm:locreg} Let $d = 2$ and $\V{u}$ be a weak solution to \eqref{linearsys} when $F = 0$. Suppose that $B_{\color{red}7R} \subseteq \Omega$ is an open ball of radius ${\color{red}7R}$ and $B = B_R$ is the concentric ball with radius $R$. Then there exists $\epsilon \approx [W]_{\text{A}_2}^{-8}$ such that for $x, y \in B$, we have  \begin{equation*} |\V{u}(x) - \V{u} (y)| \lesssim C_{x, y} {\color{red} R^{-3\epsilon} |x - y|^\epsilon} \end{equation*} where  \begin{equation*} C_{x, y} = \left(\sup \frac{1}{|B'|^{1 - \epsilon}} \int_{B'} \|W^{-\frac12}(\xi)  \|^2  \, d\xi\right)^\frac12 \end{equation*} where the supremum is over balls $B' \subset \Omega$ centered either at $x$ or $y$, and having radius $\leq 2|x - y|$. \end{theorem}  \noindent As with Theorem \ref{revmeyer}, we will also extend Theorem \ref{thm:locreg} to weak solutions of \eqref{PLap} when $F = 0$ {\color{red}(see Section \ref{regularity}).}  Note that it would be very interesting to know whether one can use Theorem \ref{thm:locreg} to prove continuity a.e. of weak solutions to \eqref{linearsys} when $F = 0$. Much more generally, it would be interesting to know whether one can modify existing but deeper techniques from the theory of elliptic systems in conjunction with our matrix weighted Poincare and Sobolev inequalities to improve upon our regularity results.

In the special case when $A_{ij}^{\alpha \beta} (x)= B_{ij} (x) \delta_{\alpha \beta} $ for some $\Mn$ valued function $B$, the system \eqref{linearsys} becomes \begin{equation} \div (B(x) D\V{u}(x)) = -(\div F)(x). \label{SpecialEllptic} \end{equation} Such systems were considered by Iwaniec/Martin \cite{IM}, Huang \cite{HU}, and Stroffolini \cite{S}. Of particular interest is when $B$ itself is a matrix A${}_2$ weight, and  Theorems \ref{revmeyer} and \ref{thm:locreg} are of independent interest themselves in this case.

The plan of the paper will be as follows.  In Section \ref{Preliminaries} we will state some notation that will be used throughout the paper.  In Section \ref{fractionalbds} we will prove Theorem \ref{thm:max} and Theorem \ref{thm:frac}.  We will prove the Poincar\'e and Sobolev inequalities in Section \ref{Poincare} and prove the existence results in Section \ref{existsec}.  Finally we finish the manuscript with the proof of the local regularity of weak solutions including the proofs of the Meyers reverse H\"older estimates (Theorem \ref{revmeyer}) and the local regularity in dimension two (Theorem \ref{thm:locreg}) in Section \ref{regularity}.  Note that despite the length of this manuscript, we have tried to make it as self contained as possible.

We will end this section by mentioning a vast family of examples of matrix A${}_2$ ``power" weights, and it is through these examples that we hope our results will find applications to concrete degenerate elliptic systems (particularly with respect to \eqref{SpecialEllptic} when $A$ is a matrix A${}_2$ ``power" weight.)  More precisely, Bickel,  Lunceford, and  Mukhtar recently proved the following in  \cite{BLM}.

 \begin{theorem} \label{BLM} Let $A = (a_{ij})$ be a positive definite  $n \times n$ matrix.

  \begin{list}{}{}
  \item $1) \ $ If $\gamma_{ij} \in \mathbb{R}$ for $i, j = 1, \ldots, n$ then a matrix weight $W$ of the form $W_{ij} (x) = a_{ij} |x|^{\gamma_{ij}}$ is a matrix A${}_2$ weight if and only if   $-d< \gamma_{ii} <d$ and each $\gamma_{ij} = (\gamma_{ii} + \gamma_{jj})/2$ for all $i, j  = 1, \ldots, n$.  \\

         \item $2) \ $ If $\gamma_{ii} ^k \in \mathbb{R}$ for $i, j = 1, \ldots, n$ and $k = 1, \ldots, d$ then a matrix weight $W$ of the form $W_{ij} (x) = a_{ij} |x_1|^{\gamma_{ij} ^1} \cdots |x_d| ^{\gamma_{ij} ^d} $ is a matrix A${}_2$ weight if and only if   $-1< \gamma_{ii} ^k < 1$, and each $\gamma_{ij} ^k = (\gamma_{ii} ^k + \gamma_{jj} ^k)/2$ for all $i, j  = 1, \ldots, n, k = 1, \ldots, d$.\end{list}
\end{theorem}

\noindent Given the results of this paper, it would clearly be interesting to extend Theorem \ref{BLM} to the $p \neq 2$ case.
\section{Preliminaries}\label{Preliminaries}

We will first need the notion of dyadic grid.  Cubes will always be assumed to have sides parallel to the coordinate axes and we will denote the side-length of a cube $Q$ as $\ell(Q)$.   A dyadic grid, usually denoted $\mathscr D$ will be a collection of cubes that satisfy the following three properties:
\begin{enumerate}
\item If $Q\in \mathscr D$ then $\ell(Q)=2^k$ for some $k\in \Z$.
\item If $\D^k=\{Q\in \D:\ell(Q)=2^k\}$, then $\Rd=\bigcup_{Q\in \D_k}Q.$
\item If $Q,P\in \D$ then $Q\cap P$ is either $\varnothing,Q,$ or $P$.
\end{enumerate}
We will use the following well known fact about dyadic grids whose proof can be found in a recent manuscript by Lerner and Nazarov \cite{LN}.
\begin{proposition}\label{dyadic} Let $\D^t = \{2^{-k} ([0, 1) ^d + m  + (-1)^k t)  : k \in \Z, m \in \Z^d\}$, then given any cube $Q$, there exists $1\leq t\leq 2^d $ and $Q_t\in \D^t$ such that $Q\subset Q_t$ and $\ell(Q_t)\leq 6\ell(Q)$.
\end{proposition}

We now establish the machinery of the matrix weights needed for the paper.  Given a cube $Q$, let $\W{V}_Q, \ \W{V}_Q ' $ be a reducing operator (i.e. a positive definite $n \times n$ matrix) where \begin{equation*} |\W{V}_Q \V{e}| \approx \left(\frac{1}{|Q|} \int_Q |W^{-\frac{1}{q}} (x) \V{e}|^{p'} \, dx \right)^\frac{1}{p'}, \ \ \ |\W{V}_Q '  \V{e}| \approx \left(\frac{1}{|Q|} \int_Q |W^\frac{1}{q} (x) \V{e}|^{q } \, dx \right)^\frac{1}{q}. \end{equation*}  In fact we can pick the reducing operators in such a way that \begin{equation*}  \left(\frac{1}{|Q|} \int_Q |W^{-\frac{1}{q}} (x) \V{e}|^{p'} \, dx \right)^\frac{1}{p'} \leq      |\W{V}_Q \V{e}| \leq \sqrt{n} \left(\frac{1}{|Q|} \int_Q |W^{-\frac{1}{q}} (x) \V{e}|^{p'} \, dx \right)^\frac{1}{p'} \end{equation*} and a similar statement holds for $|\W{V}_Q '  \V{e}|$ (see Proposition $1.2$ in \cites{G}).  Using the reducing operators we see that
\begin{align} \label{Apqcond} \sup_{Q} \, \|\W{V}_Q \W{V}_Q '\|^q  & {\color{red} = \sup_{Q} \, \|\W{V}_Q '  \W{V}_Q \|^q }
\\ & \approx [W]_{A_{p,q}} \nonumber
\\ & =\sup_{Q} \, \frac{1}{|Q|} \int_Q \left(\frac{1}{|Q|} \int_Q \|W^\frac{1}{q} (x) W^{-\frac{1}{q}}(y) \|^{p'} \, dy \right)^\frac{q}{p'} \, dx. \nonumber \end{align}

Let $\rho$ be a norm on $\Cn$ and let $\rho^*$ be the dual norm defined by   \begin{equation*} \rho^* (\V{e}) = \sup_{\V{f} \in \Cn} \frac{\left|\ip{\V{e}}{\V{f}}_{\Cn} \right|}{\rho(\V{f})}. \end{equation*}  By elementary arguments we have that $(\rho^*)^* = \rho$ for any norm $\rho$.  Also let \begin{equation*} \rho_{q, Q} (\V{e}) = \left(\frac{1}{|Q|} \int_Q |W^\frac{1}{q} (x) \V{e} |^q \, dx \right)^\frac{1}{q} \end{equation*} so that  $\rho_{q, Q} (\V{e}) \approx |\W{V}_Q ' \V{e}|$ and by trivial arguments $\rho_{q, Q} ^* (\V{e}) \approx |(\W{V}_Q ')^{-1} \V{e}|$.
%Note that in the scalar weighted setting  $W$ is an A${}_{p, q}$ weight if and only if $W \in \text{A}_r$ for $r = 1 + \frac{q}{p'}$ with $\|W\|_{\text{A}_r} = \Apq{W}$.

%Throughout we assume that \begin{equation*} \frac{1}{q} = \frac{1}{p} - \frac{\alpha}{d}. \end{equation*}

\section{Bounds for fractional operators} \label{fractionalbds}

In this section we will prove Theorems \ref{thm:max} and \ref{thm:frac}.  We begin with some facts about the matrix A$_{p,q}$ condition.  Throughout this section we will assume that $0\leq\al<d$ (and $0 < \al < d$ when dealing with fractional integral operators), and  $p, q$ satisfy the Sobolev relationship
$$\frac1q=\frac1p-\frac{\al}{d}.$$

\begin{proposition} \label{ApqReform} $W$ is an A${}_{p, q}$ weight if and only if the averaging operators \begin{equation*} \V{f} \mapsto \frac{{\color{red} \chi_Q}}{|Q|^{1 - \frac{\alpha}{d}}} \int_Q \V{f}(x) \, dx \end{equation*}  are uniformly bounded from $L^p(W^\frac{p}{q}) $ to $L^q(W)$. \end{proposition}

\begin{proof} The proof is similar to Proposition $2.1$ in \cites{G}.  In particular, since $L^p(W^{\frac{p}{q}}) $ is the dual space of $L^{p'} (W^{-\frac{p'}{q}})$ under the usual unweighted pairing \begin{equation*} L_{\V{g}} (\V{f}) = \ip{\V{f}}{\V{g}}_{L^2(\Rd; \Cn)} \end{equation*} for $\V{g} \in L^p(W^\frac{p}{q})$, we have that
\begin{align*} \lefteqn{\sup_{\|\V{f}\|_{L^p (W^\frac{p}{q})} = 1} \left\|\frac{{\color{red} \chi_Q}}{|Q|^{1 - \frac{\alpha}{d}}} \int_Q \V{f} (x)\, dx  \right\|_{L^q(W)}   = \sup_{\|\V{f}\|_{L^p (W^\frac{p}{q})} = 1} |Q|^{-\frac{1}{p'}} \rho_{q, Q}\left( \int_Q \V{f} \, dx \right)}  \\
&\hspace{5cm}  = \sup_{\|\V{f}\|_{L^p (W^\frac{p}{q})} = 1} \sup_{\V{e} \in \Cn} |Q|^{-\frac{1}{p'}} \frac{ \left|\int_Q \ip{\V{f}(x)}{\V{e}}_{\Cn} \, dx \right| }{ (\rho_{q, Q})^* (\V{e})}  \\ &\hspace{5cm}= \sup_{\V{e} \in \Cn} |Q|^{-\frac{1}{p'}} \frac{\|{\color{red} \chi_Q} \V{e} \|_{L^{p'} (W^{-\frac{p'}{q}})}}{ (\rho_{q, Q})^* (\V{e})}
\end{align*}
and the last term here being uniformly finite (with respect to all cubes $Q$) is easily seen to be equivalent to $W$ being an A${}_{p, q}$ weight by replacing $\vec{e}$ with $\W{V}_Q ' \vec{e}$.  \end{proof}
\subsection{The fractional maximal operator}
Recall that the natural definition of the maximal operator on matrix weighted spaces is given by
\begin{equation*} M_{W, \alpha}  \V{f} (x) = \sup_{Q \ni x } \frac{1}{|Q|^{1 - \frac{\alpha}{d}}} \int_Q |W^\frac{1}{q} (x) W^{-\frac{1}{q}} (y) \V{f} (y)| \, dy. \end{equation*}  We will also need the following auxiliary fractional maximal operator:
\begin{equation*} M_{W, \alpha}  ' \V{f} (x) = \sup_{Q \ni x } \frac{1}{|Q|^{1 - \frac{\alpha}{d}}}  \int_Q |\W{V}_Q ^{-1} W^{-\frac{1}{q}} (y) \V{f} (y)| \, dy. \end{equation*}

\begin{corollary} If $M_{W, \alpha}  : L^p \rightarrow L^q$ boundedly then $W$ is a matrix A${}_{p, q}$ weight. \end{corollary}

\begin{proof} For each cube $Q$ containing $x$ we have \begin{align*} \left|\frac{{\color{red} \chi_Q} (x)}{|Q|^{1 - \frac{\alpha}{d}}} \int_Q W^{\frac{1}{q}} (x) \V{f}(y) \, dy \right| & \leq \frac{{\color{red} \chi_Q} (x)}{|Q|^{1 - \frac{\alpha}{d}}} \int_Q |W^{\frac{1}{q}} (x) \V{f}(y)| \, dy \\ & \leq M_{W, \alpha} (W^\frac{1}{q} \V{f}) \end{align*} so that
\begin{align*} \sup_Q \left\|\frac{{\color{red} \chi_Q}}{|Q|^{1 - \frac{\alpha}{d}}} \int_Q \V{f}(y) \, dy  \right\|_{L^q(W)} &\leq  \left\|M_{W, \alpha} (W^\frac{1}{q} \V{f}) \right\|_{L^q} \\ & \lesssim  \|\V{f}\|_{L^p (W^\frac{p}{q})}. \end{align*} \end{proof}

\begin{corollary} If $W$ is a matrix A${}_{p, q}$ weight then for any unit vector $\V{e}$ we have that $|W^\frac{1}{q} \V{e}|^q$ is a scalar A${}_{p, q}$ weight with A${}_{p, q}$ characteristic $\lesssim \Apq{W}$. \end{corollary}

\begin{proof} Let $\phi$ be any scalar function and let $\V{f} = \phi \V{e}$.  By Proposition \ref{ApqReform}, we have that
\begin{equation*} \phi \mapsto \frac{{\color{red} \chi_Q}}{|Q|^{1 - \frac{\alpha}{d}}} \int_Q \phi(x)  \, dx \end{equation*}  are uniformly bounded from the scalar weighted space $L^p(|W^\frac{1}{q} \V{e}|^p) $ to the scalar weighted space $L^q(|W^\frac{1}{q} \V{e}|^q)$.  But Proposition \ref{ApqReform} again in scalar setting then gives us that $|W^\frac{1}{q} \V{e}|^q$ is a scalar A${}_{p, q}$ weight with A${}_{p, q}$ characteristic $\lesssim \Apq{W}$. \end{proof}

\begin{proposition} \label{ApqDual} $W$ is an A${}_{p, q}$ weight if and only if $W^{-\frac{p'}{q}}$ is an A${}_{q', p'}$ weight, and in particular \begin{equation*} \|W^{-\frac{p'}{q}} \|_{\text{A}_{q', p'}} \approx \|W\|_{\text{A}_{p, q}} ^\frac{p'}{q} \end{equation*} \end{proposition}

\begin{proof} {\color{red}The proof simply involves expanding out $\|\tilde{V}_Q \tilde{V}_Q ' \|^{p'}$ using reducing operators and noticing that \begin{equation*} \frac{1}{p'} = \frac{1}{q'} - \frac{\alpha}{d}. \end{equation*}  We leave the simple linear algebra details to the interested reader.}  \end{proof}

\begin{remark} Let $r = 1 + \frac{q}{p'}$.  These two corollaries also imply that the A${}_r$ characteristic of each $|W^{-\frac{1}{q}} \V{e}|^{p'}$ is bounded by $\Apq{W}^{r' - 1} $.

Furthermore, it is easy to see that $w$ is a scalar A${}_{p, q}$ weight if and only if $w$ is a scalar A${}_r$ weight.  In the matrix case, however, there is no reason to believe that this is true.  In particular, $W$ is a matrix A${}_r$ weight precisely when  \begin{equation*} \sup_{Q} \, \frac{1}{|Q|} \int_Q \left(\frac{1}{|Q|} \int_Q \|W^{\frac{p'}{p' + q}} (x) W^{-\frac{p'}{p' + q}}(y) \|^{\frac{p' + q}{q}} \, dy \right)^\frac{q}{p'} \, dx < \infty \end{equation*} which is unlikely to imply, or be implied by \eqref{Apqcond}.
\end{remark}

\begin{lemma} \label{IntMaxEst} If  $W$ is an A${}_{p, q}$ weight  then $\|M' _{W, \alpha} \|_{L^p \rightarrow L^q}^q  \lesssim [W]_{\text{A}_{p, q}} ^{r' - 1}$ \end{lemma}

\begin{proof} By the scalar reverse H\"{o}lder inequality for A${}_\infty$ weights and the above remark, we can pick $\epsilon \approx \Apq{W}^{1 - r' }$ where \begin{equation*} \left(\frac{1}{|Q|} \int_Q |W^{-\frac{1}{q}} (x) \V{e}|^\frac{p - \epsilon}{p - \epsilon - 1} \, dx \right)^\frac{p - \epsilon - 1}{p - \epsilon} \lesssim \left(\frac{1}{|Q|} \int_Q |W^{-\frac{1}{q}} (x) \V{e}|^{p'} \, dx \right)^\frac{1}{p'}. \end{equation*}

Let $\{\V{e}_i\}_{i = 1}^n$ be any orthonormal basis of $\Cn$ and for any fixed $y \in \Rd$ let $Q$ be a cube that contains $y$.  Then by H\"{o}lder's inequality we have that
\begin{align*} \frac{1}{|Q|^{1 - \frac{\alpha}{d}}} & \int_Q |\W{V}_Q ^{-1} W^{-\frac{1}{q}} (x) \V{f} (x)| \, dx \\ & \leq |Q|^\frac{\alpha}{d}  \left(\frac{1}{|Q|} \int_Q \|W^{-\frac{1}{q}} (x) \W{V}_Q ^{-1}\|^\frac{p - \epsilon}{p - \epsilon - 1} \, dx \right)^\frac{p - \epsilon - 1}{p - \epsilon} \left(\frac{1}{|Q|} \int_Q |\V{f}(x)| ^{p - \epsilon} \, dx \right)^\frac{1}{p - \epsilon}. \end{align*}

 By the reverse H\"{o}lder inequality, we have \begin{align} {\color{red}\label{EqRHI}}\lefteqn{\left(\frac{1}{|Q|} \int_Q \|W^{-\frac{1}{q}} (x)  \W{V}_Q ^{-1}\|^\frac{p - \epsilon}{p - \epsilon - 1} \, dx \right)^\frac{p - \epsilon - 1}{p - \epsilon}}\\
 &\qquad \approx \sum_{i = 1}^n \left(\frac{1}{|Q|} \int_Q |W^{-\frac{1}{q}} (x) \W{V}_Q ^{-1} \V{e}_i|^\frac{p - \epsilon}{p - \epsilon - 1} \, dx \right)^\frac{p - \epsilon - 1}{p - \epsilon} \nonumber \\ &\qquad \lesssim \sum_{i = 1}^n \left(\frac{1}{|Q|} \int_Q |W^{-\frac{1}{q}} (x) \W{V}_Q ^{-1} \V{e}_i|^{p'} \, dx \right)^\frac{1}{p'} \\ &\qquad \approx \sum_{i = 1}^n \| \W{V}_Q  \W{V}_Q ^{-1}\| \lesssim 1 \nonumber. \end{align}

Thus, if $M$ is the Hardy-Littlewood maximal operator then an application of H\"{o}lder's inequality gives us that \begin{align*} \lefteqn{\left(\frac{1}{|Q|^{1 - \frac{\alpha}{d}}}  \int_Q |\W{V}_Q ^{-1} W^{-\frac{1}{q}} (x) \V{f} (x)| \, dx \right)^q  \lesssim |Q|^\frac{q \alpha}{d} \left(\frac{1}{|Q|} \int_Q |\V{f}(x)| ^{p - \epsilon} \, dx \right)^\frac{q}{p - \epsilon} }\\ &\qquad = |Q|^\frac{q \alpha}{d} \left(\frac{1}{|Q|} \int_Q |\V{f}(x)| ^{p - \epsilon} \, dx \right)^\frac{q - p}{p - \epsilon} \left(\frac{1}{|Q|} \int_Q |\V{f}(x)| ^{p - \epsilon} \, dx \right)^\frac{p}{p - \epsilon} \\ &\qquad \leq |Q|^\frac{q \alpha}{d} \left(\frac{1}{|Q|} \int_Q |\V{f}(x)| ^{p} \, dx \right)^\frac{q - p}{p
} \left(M(|\V{f}| ^{p - \epsilon}) (y)  \right)^\frac{p}{p - \epsilon} \\ &\qquad = \left(\inrd |\V{f}(x)| ^{p} \, dx \right)^\frac{q - p}{p
} \left(M(|\V{f}| ^{p - \epsilon}) (y)  \right)^\frac{p}{p - \epsilon} \end{align*} since $ \frac{q\alpha}{d} - \frac{q}{p} + 1  = 0$.  Thus, the standard $L^t$ bound for $M$ with $t>1$ gives us
\begin{align*} \inrd (M_{W, \alpha}  ' \V{f} (y))^q \, dx & \lesssim \| \V{f} \|_{L^p} ^{q - p} \|M (|\V{f}|^{p - \epsilon})\|_{L^\frac{p}{p - \epsilon}} ^{\frac{p - \epsilon}{p}} \\ &  \lesssim \epsilon ^{-1} \|\V{f} \|_{L^p} ^q.  \end{align*}  \end{proof}

\begin{comment}
   By the (scalar!) reverse H\"{o}lder inequality, we can pick $\epsilon \approx \Atwo{W} ^{-1}$ where \begin{equation*} \left( \frac{1}{|B|} \int_B \|W^{-\frac12} (y) (m_B (W^{-1})) ^{-\frac{1}{2}} \|^{2 + \epsilon}\right)^\frac{1}{2 + \epsilon}  \lesssim \left( \frac{1}{|B|} \int_B \|W^{-\frac12} (y) (m_B (W^{-1})) ^{-\frac{1}{2}} \|^{2 }\right)^\frac{1}{2} \lesssim 1\end{equation*}

\noindent Thus by H\"{o}lder's inequality we have \begin{align*} M_W ' \V{f} (x) &\leq \sup_{x \in B} \left( \frac{1}{|B|} \int_B \|W^{-\frac12} (y) (m_B (W^{-1})) ^{-\frac{1}{2}} \|^{2 + \epsilon}\right)^\frac{1}{2 + \epsilon}
\left(\frac{1}{|B|} \int_B |\V{f}(y)|^{\frac{2 + \epsilon}{1 + \epsilon}} \, dy \right)^\frac{1 + \epsilon}{2 + \epsilon} \\& \lesssim (M(|\V{f}|^{\frac{2 + \epsilon}{1 + \epsilon}}) (x) )^\frac{1 + \epsilon}{2 + \epsilon} \end{align*}  where $M$ is the standard scalar dyadic maximal function.

Finally, the usual Marcinkewicz interpolation theorem estimate gives us that \begin{equation*} \inrd |M_W ' \V{f} (x)|^2 \, dx \leq \inrd (M(|\V{f}|^{\frac{2 + \epsilon}{1 + \epsilon}}) (x))^\frac{2 + 2\epsilon}{2 + \epsilon} \, dx \lesssim \epsilon ^{-1} \|\V{f} \|_{L^2} ^2 \end{equation*} which completes the proof as $\epsilon ^{-1} \approx \Atwo{W}$.  \end{proof}
\end{comment}

\begin{lemma} \label{StopLem} Let $Q$ be a dyadic cube (in some fixed dyadic lattice) and \begin{equation*} N_Q (x) = \sup_{Q \supseteq R \ni x} \|W^\frac{1}{q} (x) \W{V}_R \| \end{equation*} where the supremum is taken over all dyadic cubes $R \subseteq Q$ containing $x$.  If $W$ is an A${}_{p, q}$ weight then we can pick $\delta \approx \Apq{W}^{-1}$ where \begin{equation*} \int_Q (N_Q (x) )^{q + \epsilon} \, dx \lesssim |Q| \Apq{W}   \end{equation*}  for all $0 \leq \epsilon < \delta$. \end{lemma}

\begin{proof}

\begin{comment}
As in \cites{BPW}, one can define $W_n$ so that $W_n \rightarrow W$ and $W_n ^{-1} \rightarrow W^{-1} $ pointwise a.e.,  $\Atwo{W_n} \lesssim \Atwo{W}$ for each $n$, and $W_n, W_n^{-1} \leq n\text{Id}_{d \times d}.$ If \begin{equation*} N_I ^n  (x) = \sup_{x \in R \subset I} \|W_n ^\frac{1}{2} (x) (m_R (W_n^{-1}))  ^{\frac12}\| \end{equation*} then $\|(m_R (W_n^{-1}))  ^{-\frac12} - (m_R (W^{-1}))  ^{-\frac12}\| \rightarrow 0 $ as $n \rightarrow \infty$ so that by Fatou's lemma \begin{equation*} \int_Q (N_Q(x))^2 \, dx  \leq \int_Q (\liminf_{n \rightarrow \infty} N_Q ^n (x))^2 \, dx \leq \liminf_{n \rightarrow \infty}  \int_Q (N_Q ^n (x))^2 \, dx. \end{equation*}  Obviously $N_Q ^n (x) \leq n$ so trivially there exists $B_n$ such that \begin{equation*}
\int_Q (N^n _Q (x) )^2 \, dx \lesssim B_n |Q|   \end{equation*}
\end{comment}

\begin{comment}

1/5/2016 commented out

 By an approximation argument in \cite{BPW} p. 12, we may assume that the weights are truncated, i.e.,
 $$W, W^{-1} \leq \, k\text{Id}_{n \times n}.$$
 Our bounds will not depend on $k$ and hence we will be to obtain the desired result by a limiting argument.
 As in \cites{G}, we will show that $B \lesssim \Apq{W}$ if we assume that \begin{equation*} \int_Q (N_Q (x) )^{q + \epsilon} \, dx \lesssim B |Q|. \end{equation*} %(one can then use the ideas in \cites{BPW} to make this precise, and in particular one can use the truncated matrix weights $W_k$ from p. 7 in \cites{BPW} for $k \in \mathbb{N}$, which satisfy $W_k, W_k^{-1} \leq k \, \text{Id}_{n \times n}$.
\end{comment}
Let $Q$ be any dyadic cube and for $m \in \mathbb{N}$ let  \begin{equation*} N_{Q, m} (x) = \sup_{\substack{Q \supseteq R \ni x \\ \ell(R) > 2^{-m}}} \|W^\frac{1}{q} (x) \W{V}_R \| \end{equation*} where $\ell(R)$ denotes the side length of the cube $R$.

  Let $\{R_j\}$ be maximal dyadic subcubes of $Q$ satisfying \begin{equation*} \|\W{V}_Q ^{-1}  \W{V}_{R_j} \| > C \end{equation*} and $\ell(R_j) > 2^{-m}$ for some large $C > 1$ independent of $W$ to be determined.  Note that \begin{align*} C^{p'}   \sum_j |R_j| & \leq  \sum_j |R_j | \|\W{V}_Q ^{-1}  \W{V}_{R_j} \|^{p'}    \\ & \lesssim \sum_{i = 1}^n \sum_j  \int_{R_j} |W^{-\frac{1}{q}} (x) \W{V}_Q ^{-1}  \V{e}_i |^{p'} \, dx  \lesssim   |Q|  \end{align*} where $\ell(R)$ denotes the side length of $R$.  Thus for $C$ large enough independent of $W$ we have  $\sum_j |R_j| \leq \frac12 |Q|$ and each $R_j$ (if any even exist) satisfies $R_j \subsetneq Q$.

On the other hand if $x \in Q \backslash \cup_j R_j$ then for any dyadic cube $R \subseteq Q$ containing $x$ with $\ell(R) > 2^{-m}$ we have \begin{align*} \|W^\frac{1}{q} (x) \W{V}_R \| & \leq \|W^\frac{1}{q} (x) \W{V}_Q  \| \|  \W{V}_Q ^{-1}      \W{V}_R\|  \\ & \leq C \|W^\frac{1}{q} (x) \W{V}_Q  \|    \end{align*} so that \begin{align*}   \int_{Q \backslash \cup_j R_j} (N_{Q, m}(x))^{q + \epsilon} \, dx   & \leq C^{q + \epsilon}  \int_{Q} \|W^{\frac{1}{q}} (x) \W{V}_Q \|^{q + \epsilon} \, dx \\ & \leq C^{q + \epsilon}  \sum_{i  = 1}^n \int_{Q} |W^{\frac{1}{q}} (x) \W{V}_Q \V{e}_i |^{q + \epsilon} \, dx \\ & \leq C^{q + \epsilon}  \Apq{W} ^\frac{q + \epsilon}{q} |Q| \\ & \lesssim  C^{q + \epsilon}  \Apq{W} |Q| \end{align*} since $0 \leq \epsilon \lesssim \Apq{W}^{-1} \leq 1$.

If $x \in R_j$ and $N_{Q, m} (x) \neq   N_{{R_j, m}} (x)$ then by maximality and the arguments above we have $N_{Q, m} (x) \leq C \|W^\frac{1}{q} (x) \W{V}_Q  \|$.  So if $R = R_j$ and \begin{equation*} F_R = \{x \in R : N_{Q, m} (x) \neq   N_{R, m} (x)\} \end{equation*} then arguing as above gives us that \begin{equation*} \int_{F_R} (N_{Q, m}(x))^{q + \epsilon} \, dx \leq C^{q + \epsilon}  \Apq{W} |Q|. \end{equation*} Setting $D_1 = \{R_j\}, \tilde{C} = 2 C^{q + \epsilon}  \Apq{W},$ and combining what is above gives us that \begin{align} \int_Q  (N_{Q, m} (x) )^{q + \epsilon} \, dx  & = \left(\int_{Q \backslash \left(\bigcup_{R \in D_1} R\right)} + \sum_{R \in D_1} \int_{F_R}  + \sum_{R \in D_1} \int_{R \backslash F_R} \right) (N_{Q, m} (x) )^{q + \epsilon} \, dx \nonumber \\ & \leq \tilde{C} |Q| + \sum_{R \in D_1} \int_{R } (N_{R, m} (x) )^{q + \epsilon} \, dx \label{BaseCase}  \end{align} where $D_1$ is a (possibly empty) disjoint collection of dyadic subcubes strictly contained in $Q$ and satisfying \begin{equation*} \sum_{R \in D_1} |R|  \leq 2^{-1} |Q|. \end{equation*}

We now proceed inductively.  Clearly the Lemma is proved if $D_1 = \emptyset$.  Otherwise, for each $\tilde{R} \in D_1$ let $R = R_{\tilde{R}, j}$ be maximal dyadic subcubes of $\tilde{R}$ satisfying \begin{equation*} \|\W{V}_{\tilde{R}} ^{-1}  \W{V}_{R} \| > C \end{equation*} and $\ell(R) > 2^{-m}$. Furthermore, let $D_2 = \{R_{\tilde{R}, j} : \tilde{R} \in D_1\}$.  Then by \eqref{BaseCase} we have
\begin{align} \int_Q  (N_{Q, m} (x) )^{q + \epsilon} \, dx & \leq \tilde{C} |Q| + \sum_{R \in D_1} \int_{R } (N_{R, m} (x) )^{q + \epsilon} \, dx \nonumber  \\ & \leq \tilde{C} |Q| + \tilde{C} \sum_{R \in D_1} |R| + \sum_{R \in D_2} \int_{R } (N_{R, m} (x) )^{q + \epsilon} \, dx \nonumber \\ & \leq \tilde{C} |Q| + \frac{\tilde{C}}{2} |Q| + \sum_{R \in D_2} \int_{R } (N_{R, m} (x) )^{q + \epsilon} \, dx \label{NequalsTwo} \end{align} where \begin{equation*} \sum_{R \in D_2} |R|  \leq 2^{-2} |Q|\end{equation*}  and each dyadic cube of $D_2$ is a strict subset of a dyadic cube in $D_1$ (where again the lemma is proved if $D_2 = \emptyset$.)  Continuing like this, we obtain classes $\{D_k\}$ where each dyadic cube in $D_k$ is a strict subset of a dyadic cube in $D_{k - 1}$ and
\begin{equation*} \int_Q  (N_{Q, m} (x) )^{q + \epsilon} \, dx \leq (2 - 2^{-k}) \tilde{C}  |Q| + \sum_{R \in D_k} \int_{R } (N_{R, m} (x) )^{q + \epsilon} \, dx. \end{equation*} Let $M = \log_2 (\ell(Q))$ then by definition $D_{k}$ for $k \geq m + M$ is empty which gives us that \begin{equation*} \int_Q  (N_{Q, m} (x) )^{q + \epsilon} \, dx \leq (2 - 2^{-m - M}) \tilde{C}   |Q|. \end{equation*}   The monotone convergence theorem now completes the proof.

% Then we have\begin{equation*}  \int_{\cup_j R_j} (N_Q(x))^{q + \epsilon} \, dx = \sum_j \int_{R_j} (N_{R_j} %(x))^{q + \epsilon} \, dx \leq B \sum_j |R_j| \leq \frac{1}{2} B |Q|. \end{equation*}  Finally this implies that %there exists $C$ independent of $W$ where  $B \leq    \frac{1}{2}B + C \Apq{W} $ which completes the proof.

  \end{proof}

We are now ready to prove Theorem \ref{thm:max}

\begin{proof}[Proof of Theorem \ref{thm:max}] By Proposition \ref{dyadic} we may assume that the supremum defining $M_{W,\al}$ is over all cubes from a fixed dyadic grid $\mathscr D$.  For each $x \in \Rd$ let $R_x$ be dyadic cube containing $x$ such that \begin{align} \frac12 (M_{W, \alpha} \V{f})(x) &\leq \frac{1}{|R_x|^{1 - \frac{\alpha}{d}}} \int_{R_x}  |W^\frac{1}{q} (x) W^{-\frac{1}{q}} (y) \V{f}(y)| \, dy  \label{RxDef} \\ & \leq \|W^\frac{1}{q} (x) \W{V}_{R_x}\| \left(\frac{1}{|R_x|^{1 - \frac{\alpha}{d}}} \int_{R_x} |\W{V}_{R_x} ^{-1} W^{-\frac{1}{q}} (y) \V{f}(y)| \, dy \right). \nonumber \end{align}

 For $x \in \Rd$ pick $j \in \Z$ to satisfy   \begin{equation} 2^j \leq  \frac{1}{|R_x|^{1 - \frac{\alpha}{d}}} \int_{R_x} |\W{V}_{R_x} ^{-1} W^{-\frac{1}{q}} (y) \V{f}(y)| \, dy < 2^{j + 1} \label{SjDef} \end{equation} and let $\mathcal S_j$  be the collection of all cubes $R = R_x$ for all $x \in \Rd$ that are maximal with respect to (\ref{SjDef}) (note that H\"{o}lder's inequality implies that such a maximal cube exists).   Then, for every $x \in \Rd$ we have that $R_x \subseteq S \in \mathcal S_j$ for some $j = j_x \in \Z$ and $S \in \mathcal S_j$.  Then for such $S\in \mathcal S_j$ we have
\begin{align*}  (M_{W, \alpha} \V{f})(x)  & \leq 2 \|W^\frac{1}{q} (x) \W{V}_{R_x}\| \left(\frac{1}{|R_x|^{1 - \frac{\alpha}{d}}} \int_{R_x} |\W{V}_{R_x} ^{-1} W^{-\frac{1}{q}} (y) \V{f}(y)| \, dy \right) \nonumber  \\ & \leq 2 (2 ^{j + 1} )N_{S} (x)    \end{align*} so that finally the previous two lemmas give us that

\begin{align*} \inrd |M_{W, \alpha} \V{f} (x) |^q \, dx &\lesssim  \sum_{j \in \Z,\  S \in \mathcal S_j}  2^{qj} \int_S (N_S (x)) ^q \, dx \\ & \lesssim \Apq{W} \sum_{j \in \Z} 2 ^{qj} |\bigsqcup \mathcal S_j|  \\ & \leq \Apq{W} \sum_{j \in \Z} 2 ^{qj} |\{x : M_{W, \alpha} ' \V{f} (x) \geq 2^j\}|  \\ & \approx   \Apq{W} \|M_{W, \alpha} ' \V{f} \|_{L^q} ^q  \\ & \lesssim \Apq{W}^ {r'}\|\V{f}\|_{L^p} ^p \end{align*}  which completes the proof.
\end{proof}

\begin{comment}
\textbf{Remark}: I think it's clear (not from just this but from Goldberg and other papers) that $M_{W, \alpha} ' $  is the correct replacement for $M_{u} ^{\MC{D}}$ in your (Kabe) pdf you sent Hyun and I and the fact that these do not satisfy (strong type) weighted norm estimates for general matrix weights, is largely the reason that sharp estimates seem very hard (that and also the fact that no one seems to know if one can even prove a weighted Carleson embedding theorem for general matrix weights).

We do, however, have the following elementary but nice result (which I have no clue how to put to good use, besides possibly a clever Lorentz space duality argument?).  Perhaps this result though, could be used to prove sharp weak type estimates?
\end{comment}
We end our discussion of the fractional maximal function on matrix weighted spaces with an observation that operator $M'_{W,\al}$ defined over dyadic cubes is weak type $(p, q)$ for any matrix weight $W$.
\begin{proposition} If $W$ is any matrix weight then $M_{W, \alpha} ' $ is weak $(p, q)$ \end{proposition}

\begin{proof}
Let $\lambda > 0$ and pick maximal dyadic cubes $Q_j$ such that \begin{equation*} \frac{1}{|Q_j| ^{1 - \frac{\alpha}{d}}} \int_{Q_j} |\W{V}_{Q_j} ^{-1} W^{-\frac{1}{q}} (y) \V{f} (y) | \, dy > \lambda \end{equation*} so that  \begin{equation*}
\{x : M_{W, \alpha} ' \V{f} (x) > \lambda\} = \bigsqcup_j Q_j .\end{equation*}

However, by H\"{o}lder's inequality we have that \begin{align*} \sum_j |Q_j| & = \sum_j \frac{|Q_j| ^{q - q \frac{\alpha}{d}}}{|Q_j| ^  {q - 1 - q \frac{\alpha}{d} }}  \\ & \leq \frac{1}{\lambda^q} \sum_j \left( \frac{1}{|Q_j| ^  { 1-  \frac{1}{q} -  \frac{\alpha}{d} }}  \int_{Q_j} |\W{V}_{Q_j} ^{-1} W^{-\frac{1}{q}} (y) \V{f} (y) | \, dy \right)^q \\ & \leq \frac{1}{\lambda^q} \sum_j \left(\frac{1}{|Q_j|} \int_{Q_j} \|\W{V}_{Q_j }^{-1} W^{-\frac{1}{q}} (y)  \|^{p'} \, dy \right)^\frac{q}{p'} \left(\int_{Q_j} |\V{f} (y)|^p \, dy \right)^\frac{q}{p} \\ & \lesssim \frac{1}{\lambda^q} \left( \sum_j \int_{Q_j} |\V{f} (y)|^p \, dy \right)^\frac{q}{p} \\ & \leq \frac{1}{\lambda^q} \left(  \inrd |\V{f} (y)|^p \, dy \right)^\frac{q}{p} \end{align*} since $\frac{q}{p} > 1$. \end{proof}

Unfortunately, it is not clear whether this result can be used to sharpen any of the results in this paper with respect to the A${}_{p, q}$ characteristic.

\subsection{Fractional integral operators}

 Let $I_\alpha$ be the Riesz potential defined by \begin{equation*}
I_\alpha \V{f} (x) = \inrd \frac{\V{f}(y)}{|x - y|^{d - \alpha}} \, dy. \end{equation*}
We begin by approximating $I_\al$ by a dyadic operator.

\begin{lemma} \label{DyadicLem} Let $\D^t$ be collection of dyadic grids from Proposition \ref{dyadic} then\begin{multline*}  \left|\ip{W ^\frac{1}{q} I_\alpha W^{-\frac{1}{q}} \V{f}}{\V{g}}_{L^2} \right| \\ \lesssim \sum_{t \in \{0, \frac13\}^d} \sum_{Q \in \D^t} \frac{1}{|Q| ^{1 - \frac{\alpha}{d}}} \int_Q \int_Q \left|\ip{  W^{-\frac{1}{q}}(y) \V{f}(y)}{W ^\frac{1}{q} (x) \V{g}(x) } _{\Cn} \right| \, dx \, dy \end{multline*}  \end{lemma}

\begin{proof}  The proof requires nothing new in the matrix setting: let $| \cdot |_\infty$ be the standard $L^\infty$ norm on $\Rd$ and let $Q(x, r)$ be the ball with center $x \in \Rd$ in this norm.  Then for each $k \in \Z$ there exists $t \in \{0, \frac13 \}^d $ and $Q_t \in \D^t$ such that $Q(x, 2^k) \subset Q_t$ and \begin{equation*} 2^{k + 1} = \ell(Q(x, 2^k)) \leq \ell(Q_t) = 12 \cdot 2^k. \end{equation*}    Thus, we have \begin{align*} \lefteqn{\left|\ip{W ^\frac{1}{q} I_\alpha W^{-\frac{1}{q}} \V{f}}{\V{g}}_{L^2} \right|}\\
& \leq \inrd \sum_{k \in \Z} \int_{2^{ k - 1} \leq |x - y|_\infty < 2^k}  \frac{\left|\ip{  W^{-\frac{1}{q}}(y) \V{f}(y)}{W ^\frac{1}{q} (x) \V{g}(x) } _{\Cn} \right|}{|x - y|^{d - \alpha}} \, dy \, dx \\ & \lesssim
\inrd \sum_{k \in \Z} \sum_{t \in \{0, \frac13\}^d} \sum_{\substack{Q \in \D^t \\ 2^{k - 1} \leq \ell(Q) < 2^k}} \frac{\chi_Q(x)}{|Q|^{1 - \frac{\alpha}{d}}} \int_Q \left|\ip{  W^{-\frac{1}{q}}(y) \V{f}(y)}{W ^\frac{1}{q} (x) \V{g}(x) } _{\Cn} \right| \, dy \\ & \lesssim \sum_{t \in \{0, \frac13\}^d} \sum_{Q \in \D^t} \frac{1}{|Q| ^{1 - \frac{\alpha}{d}}} \int_Q \int_Q \left|\ip{  W^{-\frac{1}{q}}(y) \V{f}(y)}{W ^\frac{1}{q} (x) \V{g}(x) } _{\Cn} \right| \, dx \, dy. \end{align*} \end{proof}

We are now ready to prove Theorem \ref{thm:frac}.

\begin{proof}[Proof of Theorem \ref{thm:frac}]

We will show that $$W^{\frac1q}I_\al W^{-\frac1q}:L^p\ra L^q$$ which is equivalent to the boundedness $$I_\al:L^p(W^{\frac{p}{q}})\ra L^q(W).$$  By the previous Lemma it is enough to estimate \begin{multline*} \left|\ip{W ^\frac{1}{q} I_\alpha W^{-\frac{1}{q}} \V{f}}{\V{g}}_{L^2} \right| \\
\lesssim \sum_{t \in \{0, \frac13\}^d} \sum_{Q \in \D^t} \frac{1}{|Q| ^{1 - \frac{\alpha}{d}}} \int_Q \int_Q \left|\ip{  W^{-\frac{1}{q}}(y) \V{f}(y)}{W ^\frac{1}{q} (x) \V{g}(x) } _{\Cn} \right| \, dx \, dy. \end{multline*}
By a standard approximation argument we will assume that $\vec{f}$ and $\vec{g}$ are bounded functions with compact support.  Let $\D$ be a fixed dyadic grid and pick any $a > 2^{\frac{d + 1}{p - 1}}$. As in the proof of Lemma \ref{IntMaxEst}, pick $\epsilon' \approx \|W\|_{\text{A}_{p, q}} ^{1-r'}$ and let $\MC{Q} ^k$ denote the collection \begin{equation*} \MC{Q} ^k = \{Q \in \D : a^k < \left( \frac{1}{|Q|} \int_Q | \, \V{f}(y) |^{p - \epsilon'} \, dy\right)^{\frac{1}{p - \epsilon'}}  \leq a^{k + 1}\}\end{equation*} and let $\MC{S}^k$ the collection of $Q\in \D$ that are maximal with respect to the inequality \begin{equation*}  \left(\frac{1}{|Q|} \int_Q | \, \V{f}(y) |^{p - \epsilon'} \, dy\right)^{\frac{1}{p - \epsilon'}} > a^k. \end{equation*}
Finally, set $\MC{S} = \bigcup_k \MC{S}^k$.  Since for a fixed grid $\D$ \begin{align*} \sum_{Q \in \D} & \frac{1}{|Q| ^{1 - \frac{\alpha}{d}}} \int_Q \int_Q \left|\ip{  W^{-\frac{1}{q}}(y) \V{f}(y)}{W ^\frac{1}{q} (x) \V{g}(x) } _{\Cn} \right| \, dx \, dy \\ &  \leq \sum_{Q \in \D} |Q| ^{ \frac{\alpha}{d}} \left(\frac{1}{|Q|} \int_Q  |  \W{V}_Q ^{-1}  W^{-\frac{1}{q}}(y) \V{f}(y)| \, dy \right)\left( \int_Q |  \W{V}_Q  W ^\frac{1}{q} (x) \V{g}(x)| \, dx \right) \end{align*} we can estimate
\begin{align*} \lefteqn{\sum_{Q \in \D}  \frac{1}{|Q| ^{1 - \frac{\alpha}{d}}} \int_Q \int_Q \left|\ip{  W^{-\frac{1}{q}}(y) \V{f}(y)}{W ^\frac{1}{q} (x) \V{g}(x) }_{\Cn} \right| \, dx \, dy}
\\ &  \leq \sum_k \sum_{Q \in \MC{Q}^k} |Q| ^{ \frac{\alpha}{d}} \left(\frac{1}{|Q|}\int_Q  |  \W{V}_Q ^{-1}  W^{-\frac{1}{q}}(y) \V{f}(y)| \, dy \right)\left(  \int_Q |\W{V}_Q  W ^\frac{1}{q} (x) \V{g}(x)| \, dx \right)
\\ & \leq \sum_k \sum_{Q \in \MC{Q}^k} |Q| ^{ \frac{\alpha}{d}}  \left( \frac{1}{|Q|} \int_Q \|W^{-\frac{1}{q}} (y) \W{V}_Q ^{-1}\|^\frac{p - \epsilon'}{p - \epsilon' - 1} \, dy \right)^\frac{p - \epsilon' - 1}{p - \epsilon'} \\ & \qquad \times \left( \frac{1}{|Q|}\int_Q |\V{f}(y)| ^{p - \epsilon'} \, dy \right)^\frac{1}{p - \epsilon'}\left(  \int_Q |\W{V}_Q  W ^\frac{1}{q} (x) \V{g}(x)| \, dx \right)
\\ & \lesssim \sum_k \sum_{Q \in \MC{Q}^k} |Q| ^{ \frac{\alpha}{d}}   \left( \frac{1}{|Q|}\int_Q |\V{f}(y)| ^{p - \epsilon'} \, dy \right)^\frac{1}{p - \epsilon'}\left(  \int_Q |\W{V}_Q  W ^\frac{1}{q} (x) \V{g}(x)| \, dx \right)
\\ & \leq  \sum_k  a^{k + 1} \sum_{Q \in \MC{Q}^k} |Q| ^{ \frac{\alpha}{d}} \int_Q |\W{V}_Q  W ^\frac{1}{q} (x) \V{g}(x)| \, dx \\ & = \sum_k  a^{k + 1} \sum_{P \in \MC{S}^k} \sum_{\substack{Q \in \MC{Q}^k \\ Q \subset P}} |Q| ^{ \frac{\alpha}{d}}  \int_Q |\W{V}_Q  W ^\frac{1}{q} (x) \V{g}(x)| \, dx \end{align*} where in the third inequality we used {\color{red}\eqref{EqRHI}.}

We now examine the inner most sum:
\begin{align*} \lefteqn{\sum_{\substack{Q \in \MC{Q}^k \\ Q \subset P}} |Q| ^{ \frac{\alpha}{d}} \int_Q |\W{V}_Q  W ^\frac{1}{q} (x) \V{g}(x)| \, dx}
\\  & \qquad\qquad  \leq \sum_{\substack{Q \in \D \\ Q \subset P}} |Q| ^{ \frac{\alpha}{d}} \int_Q |\W{V}_Q  W
^\frac{1}{q} (x) \V{g}(x)| \, dx \\ &\qquad\qquad = \sum_{j = 0}^\infty \sum_{\substack{Q \subset P \\ \ell(Q) = 2^{-j} \ell(P)}} |Q| ^{ \frac{\alpha}{d}} \int_Q |\W{V}_Q  W ^\frac{1}{q} (x) \V{g}(x)| \, dx
\\ &\qquad\qquad = |P|^\frac{\alpha}{d} \sum_{j = 0}^\infty 2^{-j\alpha} \sum_{\substack{Q \subset P \\ \ell(Q) = 2^{-j} \ell(P)}}  \int_Q |\W{V}_Q  W ^\frac{1}{q} (x) \V{g}(x)| \, dx
\\ &\qquad\qquad \lesssim |P|^\frac{\alpha}{d}  \int_P  {N}_P (x) |   \V{g}(x)| \, dx \end{align*} where as before \begin{equation*} {N}_P  (x) =  \sup_{P \supseteq Q \ni x}  \|  W ^\frac{1}{q} (x)  \W{V}_Q\|. \end{equation*}

Plugging this back into the original sum gives us \begin{align} \sum_{Q \in \D} & \frac{1}{|Q| ^{1 - \frac{\alpha}{d}}} \int_Q \int_Q \left|\ip{  W^{-\frac{1}{q}}(y) \V{f}(y)}{W ^\frac{1}{q} (x) \V{g}(x) }_{\Cn} \right| \, dx \, dy \nonumber
\\ & \lesssim \sum_k a^{k + 1} \sum_{P \in \MC{S}^k} |P|^\frac{\alpha}{d}  \int_P  {N}_P (x) |   \V{g}(x)| \, dx \nonumber \\ & \leq a \sum_k  \sum_{P \in \MC{S}^k} |P| ^{1 + \frac{\alpha}{d}}  \left( \frac{1}{|P|}\int_P |\V{f}(y)| ^{p - \epsilon'} \, dy \right)^\frac{1}{p - \epsilon'} \nonumber
\\
&\hspace{3.5cm} \times\left(  \frac{1}{|P|}\int_{P} {N}_P (x) |   \V{g}(x)| \, dx\right).  \label{LastEst} \end{align}

However, for any $u \in P$, \begin{align*}\lefteqn{ \frac{1}{|P|} \int_{P} {N}_P(x) |\V{g}(x)| \, dx }\\ & \lesssim
\left(\frac{1}{|P|} \int_{P} ({N}_P(x)) ^\frac{q' - \epsilon}{q' - \epsilon - 1} \, dx \right)^\frac{q' - \epsilon - 1}{q' - \epsilon} \left(\frac{1}{|P|} \int_{P} |\V{g}(x)|^{q' - \epsilon} \, dy \right)^\frac{1}{q' - \epsilon} \\ & \lesssim \|W\|_{\text{A}_{p, q}}^\frac1q \left(M(|\V{g}|^{q' - \epsilon}) (u) \right)^\frac{1}{q' - \epsilon}  \end{align*}
for $\epsilon \approx \|W\|_{\text{A}_{p, q}}^{-1}$ small by Lemma \ref{StopLem}.  On the other hand, we have for $u \in P$ that
\begin{align*}  |P  | ^\frac{\alpha}{d} & \left(\frac{1}{|P|} \int_P |\V{f}(y)| ^{p - \epsilon'} \, dy \right)^\frac{1}{p - \epsilon'}
\\ &= |P|^\frac{ \alpha}{d} \left(\frac{1}{|P|} \int_P |\V{f}(y)| ^{p - \epsilon'} \, dy \right)^\frac{q - p}{q(p - \epsilon')} \left(\frac{1}{|P|} \int_P |\V{f}(y)| ^{p - \epsilon'} \, dy \right)^\frac{p}{q(p - \epsilon')}
\\ &  \leq |P|^\frac{ \alpha}{d} \left(\frac{1}{|P|} \int_P |\V{f}(y)| ^{p} \, dy \right)^\frac{q - p}{qp
} \left(M(|\V{f}| ^{p - \epsilon'}) (u)  \right)^\frac{p}{q(p - \epsilon')}
\\ & = \|\V{f}\|_{L^p} ^\frac{q - p}{q}
 \left(M(|\V{f}| ^{p - \epsilon'}) (u)  \right)^\frac{p}{q(p - \epsilon')} \end{align*} since $ \frac{q\alpha}{d} - \frac{q}{p} + 1  = 0$.  Now define $E_Q$ by \begin{equation*} E_Q = Q \backslash \bigcup_{\substack{Q' \in \MC{S} \\ Q' \varsubsetneq Q}} Q' .\end{equation*} The proof will be completed if we can show that $|E_Q| \geq \frac12 |Q|$ and $\MC{S}^m \cap \MC{S}^k$ if $k \neq m$. To see this, since $\{E_Q\}_{Q \in \MC{S}}$ is a disjoint collection of cubes, we have

%\begin{align*} \inrd (M_{W, \alpha}  ' \V{f} (y))^q \, dx & \lesssim \| \V{f} \|_{L^p} ^{q - p} \|M (|\V{f}|^{p - %\epsilon})\|_{L^\frac{p}{p - \epsilon}} ^{\frac{p - \epsilon}{p}} \\ &  \lesssim \epsilon ^{-1} \|\V{f} \|_{L^p} %^q.  \end{align*}

 \begin{align*} \eqref{LastEst} & \lesssim a \|W\|_{\text{A}_{p, q}} ^\frac1q \|\V{f}\|_{L^p} ^{\frac{q - p}{q}} \sum_k \sum_{P \in \MC{S}^k} |P| \, \inf_{u \in P} \left(M(|\V{f}| ^{p - \epsilon'}) (u)  \right)^\frac{p}{q(p - \epsilon')} \left(M(|\V{g}|^{q' - \epsilon}) (u) \right)^\frac{1}{q' - \epsilon}
 \\ & \leq a \|W\|_{\text{A}_{p, q}} ^\frac1q \|\V{f}\|_{L^p} ^{\frac{q - p}{q}}  \sum_{Q \in \MC{S}} \int_{E_Q} \, \left(M(|\V{f}| ^{p - \epsilon'}) (u)  \right)^\frac{p}{q(p - \epsilon')} \left(M(|\V{g}|^{q' - \epsilon}) (u) \right)^\frac{1}{q' - \epsilon} \, du
 \\ & \leq a \|W\|_{\text{A}_{p, q}} ^\frac1q \|\V{f}\|_{L^p} ^{\frac{q - p}{q}}   \int_{\Rd} \, \left(M(|\V{f}| ^{p - \epsilon'}) (u)  \right)^\frac{p}{q(p - \epsilon')} \left(M(|\V{g}|^{q' - \epsilon}) (u) \right)^\frac{1}{q' - \epsilon} \, du
 \\ & \leq a \|W\|_{\text{A}_{p, q}} ^\frac1q \|\V{f}\|_{L^p} ^{\frac{q - p}{q}}   \left(\int_{\Rd} \, \left(M(|\V{f}| ^{p - \epsilon'}) (u)  \right)^\frac{p}{p - \epsilon'} \, du \right) ^\frac1q \left(\int_{\Rd} \left(M(|\V{g}|^{q' - \epsilon}) (u) \right)^\frac{q'}{q' - \epsilon} \, du\right) ^\frac{1}{q'}
 \\ & \lesssim  a (\epsilon ')^{-\frac1q}  \epsilon ^{-\frac{1}{q'}} \|W\|_{\text{A}_{p, q}} ^\frac1q \|\V{f}\|_{L^p} ^{\frac{q - p}{q}} \|\V{f}\|_{L^p} ^\frac{p}{q} \|\V{g}\|_{L^{q'}}
  \\ & \lesssim  a \|W\|_{\text{A}_{p, q}} ^{\frac1q  + \frac{r' - 1}{q} + \frac{1}{q'}}  \|\V{f}\|_{L^p} \|\V{g}\|_{L^{q'}}
  \\ & \lesssim  a \|W\|_{\text{A}_{p, q}} ^{ \frac{p'}{q} \left(1 - \frac{\alpha}{d}\right) + \frac{1}{q'}}  \|\V{f}\|_{L^p} \|\V{g}\|_{L^{q'}}\end{align*}

Now pick $k$ such that $Q \in \MC{S}_k$ and let $\W{a} = a^{p - \epsilon'}.$  Without loss of generality we can assume that $0 < \epsilon' < 1$ so that $\W{a} = a^{p - \epsilon'} \geq a^{p - 1} > 2^{d + 1} > 2^d$.  By maximality we have that

\begin{equation*} E_Q = Q \backslash \Big(\bigcup_{\substack{Q' \in \MC{S}^{k + 1} \\ Q' \varsubsetneq Q }} Q' \Big). \end{equation*}

 Note that if $\W{Q}$ is the parent of $Q$ then  \begin{equation*} \frac{1}{|Q|} \int_Q | \, \V{f}(x) |^{p - \epsilon'} \, dx \leq 2^d \frac{1}{|{\color{red}\W{Q}}|} \int_{\W{Q}} | \, \V{f}(x) |^{p - \epsilon'} \, dx \leq \W{a} ^{k}  2^d < \W{a}^{k + 1} \end{equation*} so that $\MC{S}^m \cap \MC{S}^k  {\color{red} = \emptyset} $ if $m \neq k$ and \begin{align*} \Big| \bigcup_{\substack{Q' \in \MC{S}^{k + 1} \\ Q' \varsubsetneq Q }} Q' \Big| &  \leq \frac{1}{\W{a}^{k + 1}} \sum_{\substack{Q' \in \MC{S}^{k + 1} \\ Q' \varsubsetneq Q}}   \int_{Q'}  | \, \V{f}(x) |^{p - \epsilon'} \, dx
\\ & \leq \frac{1}{\W{a}^{k + 1}}    \int_{Q}  | \, \V{f}(x) |^{p - \epsilon'} \, dx
\\ & \leq \frac{\W{a} ^{k}  2^d }{\W{a}^{k + 1}} |Q|
\\ & \leq \frac12 |Q|.  \end{align*}
\end{proof}

\section{Matrix weighted Poincare and Sobolev inequalities}\label{Poincare}

We now prove our matrix weighted Poincar\'e and Sobolev inequalities.  Recall, that in the scalar case the following representation formulas hold:
$$|f(x)-f_Q|\lesssim I_1(|\nabla f|\chi_Q)(x),\qquad x\in Q, f\in C^1(\Rd)$$
and
$$|f(x)|\lesssim I_1(|\nabla f|)(x), \qquad f\in C_0^1(\R^d).$$
%For this section, let $I_1 ' $ be the modified Riesz potential given by \begin{equation*} I_1 ' \V{f} (x) = \inrd

\begin{lemma} \label{IntRep} For  $\V{f}, \V{g} \in C_0^1(\mathbb{R}^d)$, we have that \begin{equation*} \left|\langle W^\frac{1}{q} \V{f}, \V{g}\rangle_{L^2} \right| \lesssim \inrd \, \inrd \frac{\left|\left\langle (W^\frac{1}{q} (x) D \V{f} (y)) (x - y) , \V{g} (x) \right\rangle_{\Cn}\right|}{|x  - y|^d} \, dx \, dy \end{equation*}  where $D \V{f} (x) $ is the standard Jacobian matrix of $\V{f}$ at $x$.   \end{lemma}

\begin{proof} Let $\V{f} = (f_1, \cdots, f_n)$ so by standard arguments \begin{equation*}
f_i(x) = - \frac{1}{d \omega_d} \inrd \frac{\langle \nabla f_i (y), (x - y) \rangle_{\Rd}}{|x - y|^d} \, dy \end{equation*} where $\omega_d$ is the volume of the unit ball in $\Rd$ and \begin{equation*} \langle \V{u}, \V{v}\rangle_{\Rd} = \sum_{i = 1}^d u_i v_i \end{equation*} for $\V{u} \in \mathbb{C}^d$ and $\V{v} \in \Rd$.  Thus, by elementary matrix manipulations and the definition of $D\V{f}$ we have that \begin{equation*} W^\frac{1}{q} (x) \V{f}(x) =  - \frac{1}{d \omega_d} \inrd \frac{(W^\frac{1}{q} (x) D\V{f} (y))(x - y) }{|x - y|^d} \, dy \end{equation*} which implies the lemma. \end{proof}

With the help of Lemma \ref{IntRep}, the proof of the following is very similar to the proof Theorem \ref{thm:frac}, and therefore we will only sketch the details.

\begin{theorem} \label{GlobalSob}If $W $  is a matrix A$_{p, q}$ weight where \begin{equation*} \frac{1}{q} = \frac{1}{p} - \frac{1}{d} \end{equation*} then \begin{equation*} \|W^\frac{1}{q} \V{f} \|_{L^q} \lesssim \|W^\frac{1}{q} D \V{f} \|_{L^p} \end{equation*}  for Schwartz functions $\V{f}$ and $\V{g}$.  \end{theorem}

\begin{proof} The arguments in Lemma \ref{DyadicLem} and Lemma \ref{IntRep} give us that \begin{align*}\lefteqn{\left|\langle W^\frac{1}{q} \V{f}, \V{g}\rangle_{L^2} \right|}\\
 & \lesssim \inrd \, \sum_{k \in \Z} \sum_{t } \sum_{\substack{Q \in \D^t \\ 2^{k - 1} \leq \ell(Q) < 2^k}}   \chi_Q (x) \int_Q \frac{\left|\left\langle (W^\frac{1}{q} (x) D \V{f} (y)) (x - y) , \V{g} (x) \right\rangle_{\Cn}\right|}{|x  - y|^d} \, dy \, dx \\ & \lesssim \sum_{t } \sum_{Q \in \D^t} \frac{1}{|Q|} \int_Q \, \int_Q \, \left|\left\langle (W^\frac{1}{q} (x) D \V{f} (y)) (x - y) , \V{g} (x) \right\rangle_{\Cn}\right| \, dy \, dx \\ & \lesssim
\sum_{t} \sum_{Q \in \D^t} \frac{1}{|Q|^{1 - \frac{1}{d}}} \int_Q \, \int_Q \|{\W{V}}_Q ^{-1} W^{-\frac{1}{q}} (y) (W^\frac{1}{q} (y) D\V{f}(y)) \| | {\W{V}}_Q W^\frac{1}{q} (x) \V{g} (x) | \, dy \, dx.  \end{align*}

Repeating the stopping time arguments from the proof of Theorem \ref{thm:frac} to estimate the last term, we get that \begin{equation*}\left|\langle W^\frac{1}{q} \V{f}, \V{g}\rangle_{L^2} \right| \lesssim \|W^\frac{1}{q} D\V{f} \|_{L^p} \|\V{g}\|_{L^{q'}}. \end{equation*}  \end{proof}

%Note: Since these arguments are ultimately ones that reduce to norm estimates of matrix weighted maximal functions, %it is very probable that one can in fact get a gain on the left hand side (which shouldn't be surprising, as %Goldberg's results in \cites{G} all have a slight gain also, as they too rest upon norm estimates of matrix weighted %maximal functions or Goldberg's ``N" function.)

For local Poincar\'e/Sobolev inequalities with gains, let  $V_P, \ V_P ' $ be the A$_p$ reducing operators: \begin{equation*} |V_P \V{e}| \approx \left(\frac{1}{|P|} \int_P |W^{-\frac{1}{p}} (x) \V{e}|^{p'} \, dx \right)^\frac{1}{p'}, \ \ \ |V_P '  \V{e}| \approx \left(\frac{1}{|P|} \int_P |W^\frac{1}{p} (x) \V{e}|^{p } \, dx \right)^\frac{1}{p}. \end{equation*}
(Note: these operators are different then the A$_{p,q}$ reducing operators.)  Furthermore, for this section,  we let \begin{equation*} M_{W, 1} ' \V{f} (x) = \sup_{\substack{P \ni x \\ P \in \D}} \frac{1}{|P| ^{1 - \frac{1}{d}}} \int_{P} |(V_P )^{-1} W^{-\frac{1}{p}} (y) \V{f}(y)| \, dy.\end{equation*}

%we let \begin{equation*} M_{W, 1} ' \V{f} (x) = \sup_{\substack{Q \ni x \\ Q \in \D}} \frac{1}{|Q| ^{1 - \frac{1}{d}}} %\int_{Q} |(V_Q ')^{-1} W^{\frac{1}{p}} (y) \V{f}(y)| \, dy\end{equation*}

We will now need the following well known version of the Marcinkiewicz interpolation theorem.  While the constants here are probably not optimal, they will suffice for our applications (see {\color{red}\cite{B}} for a proof in the scalar case, but which easily extends to the finite dimensional setting where in {\color{red}\cite{B}} we define sgn$\vec{f}(x) = |\vec{f}(x)|^{-1} \vec{f}(x) )$ when $\vec{f}(x) \neq 0$.)

\begin{lemma} \label{OffDiagMarcin}  Let $X$ be a finite dimensional normed space and $\Omega$ some measure space.  Let $p_i, q_i$ for $i = 0, 1$ be be exponents with $1 \leq p_i \leq q_i \leq \infty, p_0 < p_1, $ and $q_0 \neq q_1$. If $T$ is subadditive on $L^{p_0} (\Omega; X) + L^{p_1} (\Omega ; X)$  with \begin{equation*} \|T\vec{f} \|_{L^{q_i, \infty}}  \leq A \|\vec{f}\|_{L^{p_i}} \end{equation*} and $p^{-1} = (1 - \theta) p_0 ^{-1} + \theta p_1 ^{-1}, \ q^{-1} = (1 - \theta) q_0 ^{-1} + \theta q_1 ^{-1}$ then \begin{equation*}{\color{red}\|T\vec{f}\|_{L^q(\Omega)} \lesssim \left[2^q A^q \max_{t \in \{0, 1\}}(p_i/p)^{q_i/p_i} |q - q_i|^{-1}\right]^\frac{1}{q} \|\vec{f}\|_{L^p(\Omega)}.}  \end{equation*}  \end{lemma}

Our next result is a matrix version of Lemma $1.1$ in \cite{FKS} for the fractional matrix weighted maximal function $M'_{W, 1}$. Before we state and prove this we will need the following simple result. \begin{lemma}   \label{InvRedEst} For any $Q \subseteq P$ and $\V{e} \in \Cn$ we have that \begin{equation*}    |(V_P) ^{-1} \V{e}| \lesssim \left(\frac{|P|}{|Q|}\right)^\frac{1}{p'} |(V_Q ) ^{-1} \V{e}|  \end{equation*} \end{lemma}

\begin{proof}   We have \begin{align*} |V_P  \V{e}| &\approx \frac{1}{|P|^\frac{1}{p'} }\left( \int_P  |W^{-\frac{1}{p} } (y) \V{e} |^{p'} \, dy \right)^\frac{1}{p'} \\ & \geq \left(\frac{1}{|P| }\right)^\frac{1}{p'} \left( \int_Q  |W^{-\frac{1}{p}} (y) \V{e} |^{p'} \, dy \right)^\frac{1}{p'} \\ & \approx \left(\frac{|Q|}{|P|}\right)^\frac{1}{p'} |V_Q \V{e}| \end{align*} which implies that \begin{equation*} \|(V_P )  ^{-1} V_Q \| = \|V_Q  (V_P ) ^{-1}\|  \lesssim \left(\frac{|P|}{|Q|}\right)^\frac{1}{p'}. \end{equation*}  \end{proof}

\begin{lemma}  \label{MatrixFKSLem} Let $W$ be a matrix A${}_p$ weight, let $p \leq d$, and let $1 \leq k \leq \frac{{\color{red}d}}{d - 1 }.$    Then there exists $q < p$ independent of $k$ such that ${\color{red}p - q} \approx \|W\|_{\text{A}_p} ^\frac{p'}{p}$   where \begin{equation*}   \left(\frac{1}{|Q|} \int_{Q} (M_{W, 1}' \V{f} (x) )^{kq^*}  \, dx \right)^\frac{1}{kq^*} \lesssim (q^* - q)^{-\frac{1}{k q^*}} |Q|^\frac{1}{d} \left(\frac{1}{|Q|} \int_Q |\V{f}(x)|^{q^*} \, dx \right)^\frac{1}{q^*}  \end{equation*} for all $\V{f}$ supported on $Q$ and all $q^* > q$.
%On the otherhand, if $p \geq \frac{d}{d - 1}$ then \begin{equation*}   \left(\frac{1}{|I|} \int_{I} (M_{W, 1}' \V{f} (x) %)^{kp'} (x) \, dx \right)^\frac{1}{kp'} \lesssim \left(\frac{1}{|I|} \int_I |\V{f}(x)|^{p'} \, dx \right)^\frac{1}{p'}.  %\end{equation*}

\end{lemma}

\begin{proof} The proof is similar to the proof of Lemma $1.1$ in \cites{FKS}.  If $Q \subseteq P$ then the previous lemma gives us that there exists $C > 1$ independent of $\V{f}$ (and in fact independent of $W$) where \begin{align}\lefteqn{\frac{1}{|P| ^{1 - \frac{1}{d}}} \int_{P} |(V_P )^{-1} W^{-\frac{1}{p}} (y) \V{f}(y)| \, dy} \nonumber  \\
&\qquad \leq C \left(\frac{|P|}{|Q|}\right)^\frac{1}{p'} \frac{1}{|P| ^{1 - \frac{1}{d}}}  \int_{Q} |(V_Q )^{-1} W^{-\frac{1}{p}} (y) \V{f}(y)| \, dy \nonumber \\ &\qquad = C \left(\frac{|P|}{|Q|}\right)^{\frac{1}{p'} - 1 + \frac{1}{d}} \frac{1}{|Q|^{1 - \frac{1}{d}}} \int_{Q} |(V_Q )^{-1} W^{-\frac{1}{p}} (y) \V{f}(y)| \, dy \nonumber  \\ &\qquad \leq  C \frac{1}{|Q|^{1 - \frac{1}{d}}} \int_{Q} |(V_Q )^{-1} W^{-\frac{1}{p}} (y)\V{f}(y)| \, dy \label{matrixFKSest}\end{align} since $\V{f}$ is supported on $Q$ and $\frac{1}{d} - 1 + \frac{1}{p'} \leq 0$ by assumption.

Let \begin{equation*} E_\lambda = \{x \in Q : M_{W, 1} ' \V{f} (x) >  \lambda\}. \end{equation*} Let $\{P_j\}$ be the maximal dyadic subcubes of $Q$ such that \begin{equation*}  \frac{1}{|P_j| ^{1 - \frac{1}{d}}} \int_{P_j} |(V_{P_j} )^{-1} W^{-\frac{1}{p}} (y) \V{f}(y)| \, dy > \frac{\lambda}{C} \end{equation*} where $C$ is above. Then by maximality and \eqref{matrixFKSest} we have $E_\lambda  \subset \bigsqcup_j P_j.$

By the comments at the beginning of the proof of Lemma \ref{IntMaxEst}, we can pick $1 < q < p$ where both $q' - p' \approx \|W\|_{\text{A}_p} ^{-\frac{p'}{p}}$ and \begin{equation*} \sup_{J \in \D} \, \frac{1}{|J|}  \int_J \|(V_{J} )^{-1} W^{-\frac{1}{p} } (y)\|^{q'} \, dy  < \infty \end{equation*} are true (and where the above supremum is independent of $W$). Note then that \begin{equation*} p - q = \frac{q' - p'}{(p' - 1)(q' - 1)} \approx \|W\|_{\text{A}_p} ^{-\frac{p'}{p}} \end{equation*}  We then have by H\"{o}lder's inequality that \begin{align*} |E_\lambda|  &\leq \sum_j |P_j| \\ & \lesssim   \frac{1}{\lambda ^{kq}} \sum_j  \frac{1}{ |P_j| ^{kq - \frac{kq}{d} - 1}} \left(\int_{P_j} |(V_{P_j}  )^{-1} W^{-\frac{1}{p}} (y) \V{f}(y)| \, dy \right)^{kq} \\ & \lesssim  \frac{1}{\lambda ^{kq}} \sum_j  |P_j| ^{1  +\frac{kq}{d} - k}  \left(\int_{P_j}  |\V{f}(y) |^q \, dy \right)^k \end{align*}  However, $1 + \frac{kq}{d} - k > 1 + \frac{k}{d} -k \geq 0$ since $k \leq \frac{d}{d - 1}$.  Thus, since $\bigsqcup_j P_j \subset Q$ and $k \geq 1$ we have \begin{equation*} \left(\frac{|E_\lambda|}{|Q|}\right)^{\frac{1}{kq}} \leq \frac{C'}{\lambda} |Q|^{ \frac{1}{d} } \|\V{f}\|_{L^q(Q, \frac{dx}{|Q|})} \end{equation*} which means that $M_{W, 1} ' $ is bounded from $L^q (Q, \frac{dx}{|Q|})$ into $L^{qk,\infty} (Q,  \frac{dx}{|Q|})$ with \begin{equation*} \|M_{W, 1} '  \V{f} \|_{L^{qk, \infty}(Q, \frac{dx}{|Q|})} \lesssim |Q|^\frac{1}{d} \|\V{f}\|_{L^q (Q, \frac{dx}{|Q|})}. \end{equation*}   A similar argument shows that  ${\color{red} M_{W, 1} ' }$ is bounded from $L^{\tilde{p}} (Q, \frac{dx}{|Q|})$ into weak-$L^{{\tilde{p}}k} (Q,  \frac{dx}{|Q|})$ with \begin{equation*} \|M_{W, 1} '  \V{f} \|_{L^{ \tilde{p}k, \infty}(Q, \frac{dx}{|Q|})} \lesssim |Q|^\frac{1}{d} \|\V{f}\|_{L^{\tilde{p}} (Q, \frac{dx}{|Q|})}. \end{equation*} for all $\tilde{p} > q$.

Finally, picking $\tilde{p} = 2 q^* - q > q^*$, we can pick $\theta$ where $ (q^*)^{-1} = \theta q^{-1} + (1 - \theta)\tilde{p} ^{-1}$. Setting $q_0 = kq, q_1 = k\tilde{p}, p_0 = q, $ and $p_1 = \tilde{p}$ so that $(q^*)^{-1} = \theta (p_0)^{-1} + (1 - \theta) (p_1)^{-1}$ and $(kq^*)^{-1} = \theta (q_0)^{-1} + (1 - \theta) (q_1)^{-1}  $ completes the proof.

\begin{comment}
Now assume that $p < \frac{d}{d - 1}$ so that $p  < d$.  Pick $q < p$ as above and set $k_1 =  \frac{d}{d - q}$.  Then if $\{I_j\}$ is the collection of maximal dyadic cubes (not necessarily subcubes of $I$) where \begin{equation*}  \frac{1}{|I_j| ^{1 - \frac{1}{d}}} \int_{I_j} |(V_{I_j} ')^{-1} W^{-\frac{1}{p} (y)} \V{f}(y)| \, dy > {\lambda} \end{equation*} then clearly $|E_\lambda| = \sum_j |I_j|$.  Then by our choice of $k_1$, the arguments above give us that

\begin{equation*} |E_\lambda|^{k_1 q}   \leq \|\V{f}\|_{L^q(I)} = |I|^{\frac{1}{kq} + \frac{1}{d} - \frac{1}{q}} \|\V{f}\|_{L^q(I)}. \end{equation*}

\noindent Again as before, pick $\tilde{p} > p$ where $\tilde{p} < d$ and set $k_2  = \frac{d}{d - \tilde{p}}$ so again
\begin{equation*} |E_\lambda|^{k_2 \tilde{p}}  = (\sum_j |E_j|)^{k_2 \tilde{p}} \leq \|\V{f}\|_{L^{\tilde{p}}(I)} = |I|^{\frac{1}{k_2 \tilde{p}} + \frac{1}{d} - \frac{1}{\tilde{p}}} \|\V{f}\|_{L^{\tilde{p}}(I)}.   \end{equation*}  Another application of the Marcinkewicz interpolation theorem then says that \begin{equation*} \|M_{W, 1} \V{f}\|_{L^{k_2 p^*} (I , \frac{dx}{|I|})} \lesssim |I|^{\frac{1}{d}} \|V{f}\|_{L^p(I, \frac{dx}{|I|})} \end{equation*} for some $p^* < p$ where $p^* \rightarrow p$ as $\tilde{p} \rightarrow p$, which completes the proof.
\end{comment}

\end{proof}

We will need one more lemma before we prove Theorem \ref{MainThmSob}

\begin{lemma} \label{StoppingTimeLemma} Let $a > 0$ and for a measurable $\V{f}$ let $\MC{Q} ^k$ denote the collection \begin{equation*} \MC{Q} ^k = \{P \in \D : a^k < \frac{1}{|P|} \int_P \|V_P ^{-1}    W^{- \frac{1}{p}}  (y) \, \V{f}(y) \| \, dy \leq a^{k + 1}\}.\end{equation*} Let $\MC{S}^k$ be the collection of $P\in \D$ that are maximal with respect to the inequality  \begin{equation*}\frac{1}{|P|} \int_P \|V_P ^{-1}    W^{- \frac{1}{p}}  (y) \, \V{f}(y) \| \, dy > a^k\end{equation*} and set $\MC{S} = \bigcup_k \MC{S}^k$.  We can then pick $a \approx \Ap{W} ^\frac{1 + p'}{p}$  such that \begin{list}{}{}
\item $ 1) \, \MC{S} ^k \cap \MC{S}^{k'} = \emptyset$ if $k \neq k'$,
\item $ 2) \, $ If $Q\in \MC{S}$ and if $E_Q$ is defined by \begin{equation*} E_Q = Q \backslash \bigcup_{\substack{Q' \in \MC{S} \\ Q' \varsubsetneq Q}} Q', \end{equation*}
then $|E_Q|\geq\frac12|Q|$.
\end{list}
\end{lemma}

\begin{proof}  Note that if $Q \in \MC{S}^k$ and $\W{Q}$ is the parent of $Q$ then for any $\V{e} \in \C$ we have \begin{equation*} |{V}_{Q} ^{-1} \V{e} | \leq |{V}_{Q} ^{'} \V{e} | \lesssim  |{V}_{\W{Q}} ^{'} \V{e} | \leq \Ap{W}^\frac{1}{p} |{V}_{\W{Q}} ^{-1} \V{e} |.\end{equation*}

Thus, by maximality, \begin{equation} \frac{1}{|Q|} \int_Q |{V}_Q ^{-1}    W^{- \frac{1}{q}}  (y) \, \V{f}(y) | \, dy \lesssim \Ap{W}^\frac{1}{p}  \frac{1}{|\W{Q}|} \int_{{\color{red}\W{Q}}} |{V}_{\W{Q}} ^{-1}    W^{- \frac{1}{p}}  (y) \, \V{f}(y) | \, dy  \leq  \Ap{W}^\frac{1}{p} a^k \leq a^{k + 1}.  \label{lastest} \end{equation} if $a \gtrsim \Ap{W}^\frac{1}{p}$, which clearly proves $1)$.

As for $2)$, by maximality we have that

\begin{equation*} E_Q = Q \backslash \Big(\bigcup_{\substack{Q' \in \MC{S}^{k + 1} \\ Q' \varsubsetneq Q }} Q' \Big). \end{equation*}

\noindent  We now break up the disjoint collection $\{Q' \in \MC{S}^{k + 1}: Q' \varsubsetneq Q \}$ into two disjoint collections via the ``stopping time" argument from Lemma \ref{StopLem}.  In particular let $A_{Q, k}$ be those cubes {\color{red} $ Q'  \in \MC{S}^{k + 1}$ such that $Q' \varsubsetneq Q$ and $\|{V}_{Q'} ^{-1} {V}_{Q}  \| > a'$}, so that by the proof of Lemma \ref{StopLem} we have that $|\cup A_{Q, k}| \leq \frac14 |Q|$ for $a' \gtrsim \Ap{W}^{\frac{p'}{p}}$. Therefore, (\ref{lastest}) implies that there exists $C > 0$ independent of $W$ where

\begin{align*} \Big| \bigcup_{\substack{Q' \in \MC{S}^{k + 1} \\ Q' \varsubsetneq Q }} Q' \Big| &  \leq |\cup A_{Q, k}|  \\ & {\color{red}   \qquad  + \frac{1}{a^{k + 1}} \sum_{\substack{Q' \in \MC{S}^{k + 1} \\ Q' \varsubsetneq Q, \ Q' \nsubseteq \bigcup A_{Q, k}}}}   \int_{Q'}  |{V}_{Q'} ^{-1}   W^{-\frac{1}{p}}  (x) \, \V{f}(x) | \, dx  \\ & \leq \frac14 |Q| +  \frac{a'}{a^{k + 1}}  \sum_{\substack{Q' \in \MC{S}^{k + 1} \\ Q' \varsubsetneq Q, \ Q' \nsubseteq \bigcup A_{Q, k}}}   \int_{Q'}  |{V}_{Q} ^{-1}   W^{-\frac{1}{p}}  (x) \, \V{f}(x) | \, dx \\ & \leq  \frac14 |Q| +  \frac{a'}{a^{k + 1}} \int_{Q}  |{V}_{Q} ^{-1}   W^{-\frac{1}{p}}  (x) \, \V{f}(x) | \, dx \\ & \leq \frac14 |Q| +   \frac{[W]_{\text{A}_p} ^\frac{1}{p} a^k a'}{a^{k + 1}} |Q| \\ & \leq \frac12 |Q|.  \end{align*}  if $a = 4a' [W]_{\text{A}_p} ^{\color{red} \frac{1}{p}} \approx \Ap{W}^\frac{1+p'}{p}$. \end{proof}

%\section{Matrix weighted Poincar\'e and Sobolev Inequalities}
We are now ready to prove our matrix weighted Poincar\'e and Sobolev inequalities with gains, namely Theorem \ref{MainThmSob}.

\begin{proof}[Proof of Theorem \ref{MainThmSob}]  We first assume $p \leq d$.  The case that $p > d$ will be handled later by a duality argument. Pick  $\tilde{q} > p$ (to be determined momentarily).  Let $\V{g} \in L^{\tilde{q} '}$, and assume $\V{f}$  and $\V{g}$ are supported on $Q$. As in Theorem \ref{GlobalSob} we have that \begin{align*}\lefteqn{ \left|\langle W^\frac{1}{p} \V{f}, \V{g}\rangle_{L^2} \right|} \\
& \lesssim
 \sum_{t \in \{0, 1/3\}^d} \sum_{I \in \D^t} \frac{1}{|I|} \int_I \, \int_I \, \left|\left\langle (W^\frac{1}{p} (x) D \V{f} (y)) (x - y) , \V{g} (x) \right\rangle_{\Cn}\right| \, dy \, dx \end{align*} Now fix a dyadic grid $\D$.  Assume that $Q \in \D$.  Further, assume that $I \in \D(Q)$ since \begin{align}  \label{SumOut} \sum_{\substack{ I \in \D \\ I  \supseteq Q}} & \frac{1}{|I|} \int_I \, \int_I \, \left|\left\langle (W^\frac{1}{p} (x) D \V{f} (y)) (x - y) , \V{g} (x) \right\rangle_{\Cn}\right| \, dy \, dx  \\ & \lesssim  \frac{1}{|Q|} \int_Q \, \int_Q \, \left|\left\langle (W^\frac{1}{p} (x) D \V{f} (y)) (x - y) , \V{g} (x) \right\rangle_{\Cn}\right| \, dy \, dx \nonumber   \\ & \lesssim  \frac{1}{|Q|^{1 - \frac{1}{d}}} \int_Q \, \int_Q \|{V}_Q ^{-1} W^{-\frac{1}{p}} (y) (W^\frac{1}{p} (y) D\V{f}(y)) \| | {{V}}_Q W^\frac{1}{p} (x) \V{g}(x)| \, dx \, dy \nonumber \end{align} which will be easily estimated later in the proof.

 Let $F(y)  = W^\frac{1}{p} (y) D\V{f} (y)$. Employing the notation in Lemma \ref{StoppingTimeLemma} (but with $\V{f}$ replaced by $F$) and setting $a \gtrsim \Ap{W}^\frac{1 + p'}{p}$ we estimate

 \begin{align*}
 \sum_{I \in \D(Q)}  & \frac{1}{|I|^{1 - \frac{1}{d}}} \int_I \, \int_I \|{V}_I ^{-1} W^{-\frac{1}{p}} (y) F(y) \| | {{V}}_I W^\frac{1}{p} (x) \V{g}(x)| \, dx \, dy  \\   & =  \sum_k  \sum_{I \in \MC{Q} ^k \cap \D(Q)} \frac{1}{|I|^{1 - \frac{1}{d}}} \int_I \, \int_I \|{V}_I ^{-1} W^{-\frac{1}{p}} (y) F(y) \| | {{V}}_I W^\frac{1}{p} (x) \V{g}(x)| \, dx \, dy  \\ & \leq   \sum_k a^{k + 1} \sum_{I \in \MC{Q} ^k \cap \D(Q)} |I| ^\frac{1}{d} \int_I | {{V}}_I W^\frac{1}{p} (x) \V{g}(x)| \, dx  \\ &  \leq  \sum_k a^{k + 1} \sum_{P \in {\MC{S}} ^k} \sum_{I \in \D(P) \cap \D(Q)} |I| ^\frac{1}{d} \int_I | {{V}}_I W^\frac{1}{p} (x) \V{g}(x)| \, dx. \numberthis \label{FKSThmEst1}
 \end{align*}

We now break up (\ref{FKSThmEst1}) into two sums corresponding to $P \subseteq Q$ and $P  \supset Q$.  In the later case, note that

\begin{align*} \lefteqn{\sum_{I \in \D(Q) \cap \D(P)} |I| ^\frac{1}{d} \int_I | {{V}}_I W^\frac{1}{p} (x) \V{g}(x)| \, dx}  \\ & \qquad  = \sum_{I \in \D(Q) } |I| ^{ \frac{1}{d}} \int_I |V_I  W ^\frac{1}{p} (x) \V{g}(x)| \, dx \\ & = \sum_{j = 0}^\infty \sum_{\substack{I \subset Q \\ \ell(I) = 2^{-j} \ell(Q)}} |I| ^{ \frac{1}{d}} \int_I | V_I  W ^\frac{1}{p} (x) \V{g}(x)| \, dx \\ & = |Q|^\frac{1}{d} \sum_{j = 0}^\infty 2^{-j} \sum_{\substack{I \subset Q \\ \ell(I) = 2^{-j} \ell(Q)}}  \int_I | V_I  W ^\frac{1}{p} (x) \V{g}(x)| \, dx
\\ & \lesssim |Q|^\frac{1}{d}  \int_Q  {N}_Q (x) |   \V{g}(x)| \, dx \numberthis \label{TilingEst} \end{align*}

\noindent where as before \begin{equation*} {N}_Q  (x) =  \sup_{Q \supseteq I \ni x}  \|  W ^\frac{1}{p} (x)  V_I\|. \end{equation*}  Thus,

\begin{align*}\lefteqn{ \sum_k a^{k + 1}  \sum_{\substack{ P \in {\MC{S}} ^k \\ P  \supset Q}} \sum_{I \in \D(P) \cap \D(Q)} |I| ^\frac{1}{d} \int_I | {{V}}_I W^\frac{1}{p} (x) \V{g}(x)| \, dx}  \\ & \lesssim \sum_k a^{k + 1} \sum_{\substack{ P \in {\MC{S}} ^k \\ P  \supset Q}} |Q|^\frac{1}{d}  \int_Q  {N}_Q (x) |   \V{g}(x)| \, dx  \\
&  \leq a\sum_k  \sum_{\substack{ P \in {\MC{S}} ^k \\ P  \supset Q}} |Q|^\frac{1}{d}   \left(\frac{1}{|P|} \int_P \|V_P ^{-1}    W^{- \frac{1}{p}}  (y) \, F(y) \| \, dy\right)\left( \int_Q  N_Q(x) | \V{g}(x)| \, dx \right) \\
& \lesssim a\sum_k  \sum_{\substack{ P \in {\MC{S}} ^k \\ P  \supset Q}} |Q|^\frac{1}{d}  \left(\frac{|P|}{|Q|}\right)^\frac{1}{p'}   \left(\frac{1}{|P|} \int_Q \|V_Q ^{-1}    W^{- \frac{1}{p}}  (y) \, F(y) \| \, dy\right) \\
&\qquad \times\left( \int_Q  N_Q(x) | \V{g}(x)| \, dx \right)  \\
&  \lesssim a\sum_k  \sum_{ \substack{P \in {\MC{S}} ^k \\ P  \supset Q}} |Q|^\frac{1}{d}  \left(\frac{|P|}{|Q|}\right)^{\frac{1}{p'} - 1}   \left( \int_Q \|V_Q ^{-1}    W^{- \frac{1}{p}}  (y) \, F(y) \| \, dy\right) \\
&\qquad \times \left(\frac{1}{|Q|} \int_Q  N_Q(x) | \V{g}(x)| \, dx \right) \\
& \lesssim  a |Q|     \left(\frac{1}{|Q|^{1 -\frac{1}{d}}} \int_Q \|V_Q ^{-1}    W^{- \frac{1}{p}}  (y) \, F(y) \| \, dy\right)\left(\frac{1}{|Q|} \int_Q  N_Q(x) | \V{g}(x)| \, dx \right) \numberthis \label{FKSThmEst2} \end{align*} by Lemma \ref{InvRedEst} and the fact that $\MC{S}^k\cap\MC{S}^{k'}=\varnothing$ if $k\not=k'$ which implies that \begin{equation*} \sum_k \sum_{ \substack{P \in {\MC{S}} ^k \\ P  \supset Q}} \left(\frac{|P|}{|Q|}\right)^{\frac{1}{p'} - 1} < \infty \end{equation*}

By   Lemma \ref{StopLem}  we can pick $\tilde{q} > p$ and $C_1 > 0$ small (to be determined more precisely later) but independent of $W$ where  if $\epsilon = \tilde{q} - p$ then \begin{equation*} \sup_{P \in \D} \, \frac{1}{|P|} \int_{P} ({N}_P (x) )^{\tilde{q} + \epsilon} \, dx  \lesssim [W]_{\text{A}_p} \end{equation*} and \begin{equation*} \epsilon =  \frac{C_1}{\Ap{W}}. \end{equation*}  Then we have by H\"{o}lder's inequality

\begin{align*} \frac{1}{|Q|} \int_{Q} {N}_Q (x) |   \V{g}(x)| \, dx & \leq
\left(\frac{1}{|Q|} \int_Q (N_Q(x))^{\tilde{q} + \epsilon} \, dx \right)^\frac{1}{\tilde{q} + \epsilon} \left(\frac{1}{|Q|} \int_Q |\V{g}(x)|^{\frac{\tilde{q} + \epsilon}{\tilde{q} + \epsilon - 1}} \, dx\right)^\frac{\tilde{q} + \epsilon - 1}{\tilde{q} + \epsilon} \\ &
 \lesssim \Ap{W}^\frac{1}{p} \inf_{u \in Q} \left(M(|\V{g}|^{\frac{\tilde{q} + \epsilon}{\tilde{q} + \epsilon - 1}}) (u) \right)^\frac{\tilde{q} + \epsilon - 1}{\tilde{q} + \epsilon}.\end{align*}

Thus, \begin{equation*} (\ref{FKSThmEst2}) \lesssim  a \Ap{W}^\frac{1}{p} \int_Q (M_{W, 1} ' F)(u) \left(M(|\V{g}|^{\frac{\tilde{q} + \epsilon}{\tilde{q} + \epsilon - 1}}) (u) \right)^\frac{\tilde{q} + \epsilon - 1}{\tilde{q} + \epsilon} \, du. \end{equation*}

On the other hand, if $P \subseteq Q$ in (\ref{FKSThmEst1}) then using \eqref{TilingEst} we estimate \begin{align*} \sum_k & a^{k + 1} \sum_{\substack{ P \in {\MC{S}} ^k \\ P \subseteq Q}} \sum_{I \in \D(P) } |I| ^\frac{1}{d} \int_I | {{V}}_I W^\frac{1}{p} (x) \V{g}(x)| \, dx  \\ & \lesssim \sum_k a^{k + 1} \sum_{\substack{ P \in {\MC{S}} ^k \\ P \subseteq Q}}  |P|^\frac{1}{d} \int_P N_P (x) |\V{g}(x)| \, dx \\ & \lesssim  a \sum_k   \sum_{\substack{ P \in \MC{S}^k \\ P \subseteq Q}} |P| \left( \frac{1}{|P|^{1 - \frac{1}{d}}}  \int _P   \|{V}_P ^{-1}      W^{- \frac{1}{p}}  (y)\, F(y)\|  \, dy \right) \left( \frac{1}{|P|} \int_{P} {N}_P (x) |   \V{g}(x)| \, dx\right) \\ & \lesssim a \sum_k  \sum_{\substack{ P \in \MC{S}^k \\ P \subseteq Q}} |E_P| \left( \frac{1}{|P|^{1 - \frac{1}{d}}}  \int _P   \|{V}_P ^{-1}      W^{- \frac{1}{p}}  (y)\, F(y)\|  \, dy \right) \left( \frac{1}{|P|} \int_{P} {N}_P (x) |   \V{g}(x)| \, dx\right) \\ & \lesssim a \Ap{W}^\frac{1}{p} \int_Q (M_{W, 1} '  F)(u)  \left(M(|\V{g}|^{\frac{\tilde{q} + \epsilon}{\tilde{q} + \epsilon - 1}}) (u) \right)^\frac{\tilde{q} + \epsilon - 1}{\tilde{q} + \epsilon} \, du \end{align*} by the sparseness of the family $\{E_P\}$.

 Furthermore, by H\"{o}lder's inequality and the reverse H\"{o}lder's inequality we have \begin{align*} \eqref{SumOut} & \lesssim \frac{a}{|Q|^{1 - \frac{1}{d}}} \int_Q \int_Q \|{V}_Q ^{-1} W^{-\frac{1}{p}} (y) (W^\frac{1}{p} (y) D\V{f}(y)) \| | {{V}}_Q W^\frac{1}{p} (x) \V{g}(x)| \, dx \, dy \\ & \leq a \Ap{W}^\frac{1}{p} |Q| \inf_{u \in Q} (M_{W, 1} '  F)(u)  \left(M(|\V{g}|^{\frac{\tilde{q} + \epsilon}{\tilde{q} + \epsilon - 1}}) (u) \right)^\frac{\tilde{q} + \epsilon - 1}{\tilde{q} + \epsilon} \\ & \leq a \Ap{W}^\frac{1}{p} \int_Q (M_{W, 1} '  F)(u)  \left(M(|\V{g}|^{\frac{\tilde{q} + \epsilon}{\tilde{q} + \epsilon - 1}}) (u) \right)^\frac{\tilde{q} + \epsilon - 1}{\tilde{q} + \epsilon} \, du \end{align*}

\noindent Thus, we have (plugging back in for $F$) in the case that $Q \in \D$ \begin{equation*}  \left|\langle W^\frac{1}{p} \V{f}, \V{g}\rangle_{L^2} \right| \lesssim    a \Ap{W}^\frac{1}{p} \int_Q \left(M_{ W, 1} ' ( W^\frac{1}{p}  D\V{f}) (u)    \right) \left(M(|\V{g}|^{{\frac{\tilde{q} + \epsilon}{\tilde{q} + \epsilon - 1}}}) (u) \right)^{\frac{\tilde{q} + \epsilon - 1}{\tilde{q} + \epsilon}} \, du.  \end{equation*}

\noindent   But then another application of H\"{o}lder's inequality and the standard ``$L^{1 + \delta}$" bound for the unweighted maximal operator gives us that \begin{align*} \left|\langle W^\frac{1}{p} \V{f}, \V{g}\rangle_{L^2} \right| & \lesssim a \Ap{W}^\frac{1}{p} \|M_{W, 1} '  (W^\frac{1}{p}  D\V{f} )    \|_{L^{\tilde{q}}(Q)} \left\| \left(M(|\V{g}|^{\frac{\tilde{q} + \epsilon}{\tilde{q} + \epsilon - 1}})  \right)^{\frac{\tilde{q} + \epsilon - 1}{\tilde{q} + \epsilon}} \right\|_{L^{\tilde{q}'}(Q)} \\ & \lesssim a \Ap{W}^\frac{1}{p} \epsilon^{- \frac{1}{\W{q} '}} {\color{red}\cancel{\Ap{W}^\frac{1}{p}}} \|M_{ W, 1} '  (W^\frac{1}{p}  D\V{f} )    \|_{L^{\tilde{q}}(Q)} \|\V{g}\|_{L^{\tilde{q}'}(Q)}. \end{align*}

Finally, pick $C_1, C_2 > 0$ small enough (independent of $W$) so that \begin{equation*} k = \frac{p + \frac{C_1}{\Ap{W}}}{p - \frac{C_2}{\Ap{W}^\frac{p'}{p}}} < \frac{d}{d-1} \end{equation*} and \begin{equation*}q = p - \frac{2 C_2}{\Ap{W}^\frac{p'}{p}} {\color{red} < }     q^* = p - \frac{C_2}{\Ap{W}^\frac{p'}{p}} \end{equation*} is close enough to $p$ so that  \begin{equation*}   \left(\frac{1}{|Q|} \int_{Q} (M_{W, 1}' (W^\frac{1}{p}  D\V{f} ) (x) )^{k q^*}  \, dx \right)^\frac{1}{k q^*} \lesssim \Ap{W}^{\frac{p'}{p^2}} |Q|^\frac{1}{d} \left(\frac{1}{|Q|} \int_Q \|(W^\frac{1}{p}  D\V{f} )(x) \|^{q^*} \, dx \right)^\frac{1}{q^*}.  \end{equation*} by Lemma \ref{MatrixFKSLem}.   But then as $kq^* = \tilde{q}$  we  have \begin{align*} \left|\langle W^\frac{1}{p} \V{f}, \V{g}\rangle_{L^2} \right| & \lesssim a \Ap{W}^\frac{1}{p} \epsilon^{-\frac{1}{p'}} \|M_{W, 1} '  (W^\frac{1}{p}  D\V{f} )    \|_{L^{\tilde{q}}(Q)} \|\V{g}\|_{L^{\tilde{q}'}(Q)} \\  & = a \Ap{W}^\frac{1}{p} \epsilon^{- \frac{1}{p'}} |Q|^\frac{1}{\tilde{q}} \|M_{W, 1} '  (W^\frac{1}{p}  D\V{f} )    \|_{L^{\tilde{q}}(Q,  \frac{dx}{|Q|})} \|\V{g}\|_{L^{\tilde{q}'}(Q)} \\ & \lesssim a \Ap{W}^\frac{1}{p} \epsilon^{-\frac{1}{p'}} \Ap{W}^{\frac{p'}{p^2}} |Q|^{\frac{1}{d} + \frac{1}{\tilde{q}} - \frac{1}{q^*}} \|W^\frac{1}{p}  D\V{f}     \|_{L^{q^*}(Q)} \|\V{g}\|_{L^{\tilde{q} '}(Q) }. \end{align*}

\noindent so that by H\"{o}lder's inequality \begin{align*}
\left(\frac{1}{|Q|} \int_Q |W^\frac{1}{p} (x) \V{f}(x) |^{p+\epsilon'} \, dx \right)^{\frac{1}{p+\epsilon' }} &  \lesssim
\Ap{W} ^{1 + \frac{1}{p} + \frac{p'}{p} + \frac{p'}{p^2}} |Q|^\frac{1}{d} \left(\frac{1}{|Q|} \int_Q |W^\frac{1}{p} (x) D \V{f}(x) |^{p-\epsilon' } \, dx \right)^\frac{1}{p -\epsilon'}
\\ &  \lesssim
\Ap{W} ^{1 + \frac{2p'}{p}} |Q|^\frac{1}{d} \left(\frac{1}{|Q|} \int_Q |W^\frac{1}{p} (x) D \V{f}(x) |^{p-\epsilon' } \, dx \right)^\frac{1}{p -\epsilon'}  \end{align*}

\noindent for $\epsilon' \approx [W]_{\text{A}_p} ^{-\max\{1, \frac{p'}{p}\}}$, which completes the proof when $p \leq d$ if $Q \in \D$.

If $Q \not \in \D$, then we can pick disjoint cubes $Q_j \in \D$ for $j = 1, \ldots, 2^d$ with $\ell(Q) \leq  \ell(Q_j) < 2 \ell(Q)$ and $Q \subseteq \sqcup_j Q_j$.  Slightly abusing notation and writing $(D\V{f})_j = \chi_{Q_j} D\V{f}$ and defining $\V{g_j}$ similarly, we then obviously have \begin{align*}  \sum_{I \in \D} & \frac{1}{|I|} \int_I \, \int_I \, \left|\left\langle (W^\frac{1}{p} (x) D \V{f} (y)) (x - y) , \V{g} (x) \right\rangle_{\Cn}\right| \, dy \, dx \\ & \leq \sum_{i, j  = 1} ^{2^d}  \sum_{I \in \D} \frac{1}{|I|} \int_I \, \int_I \, \left|\left\langle (W^\frac{1}{p} (x) (D \V{f})_i (y)) (x - y) , \V{g_j} (x) \right\rangle_{\Cn}\right| \, dy \, dx. \end{align*}  If $i \neq j$ then obviously \begin{align*} \sum_{I \in \D} \frac{1}{|I|} & \int_I \, \int_I \, \left|\left\langle (W^\frac{1}{p} (x) D \V{f_i} (y)) (x - y) , \V{g_j} (x) \right\rangle_{\Cn}\right| \, dy \, dx \\ & = \sum_{\substack {I \in \D \\ I \supseteq Q_i \cup Q_j}} \frac{1}{|I|} \int_{Q_i} \, \int_{Q_j} \, \left|\left\langle (W^\frac{1}{p} (x) (D \V{f})_i (y)) (x - y) , \V{g_j} (x) \right\rangle_{\Cn}\right| \, dy \, dx \\ & \lesssim  \frac{1}{|Q|} \int_{Q} \, \int_Q \, \left|\left\langle (W^\frac{1}{p} (x) D \V{f} (y)) (x - y) , \V{g} (x) \right\rangle_{\Cn}\right| \, dy \, dx \end{align*} which was already estimated.  Finally if $i = j$ then this reduces to the proof above when $\V{f}$ and $\V{g}$ are supported on the same dyadic cube in $\D$, which completes the proof when $p \leq d$.

Now if $p > d$ then clearly $p' < \frac{d}{d - 1} \leq d$ since $d \geq 2$.  As was earlier mentioned, we have that $W$ is a matrix A${}_p$ weight if and only if $\W{W} = W ^{- \frac{p'}{p}}$ is a matrix A$_{p'}$.  Furthermore, writing $V_I (W, p), V_I'(W, p)$ to temporarily denote the matrix weight and exponent that the corresponding reducing operator is with respect to, a direct computation shows that we may let $V_I( \W{W}, p') = V_I '(W, p)$ and $V_I' (\W{W}, p')  = V_I (W, p)$ and moreover by Proposition \ref{ApqDual} that \begin{equation*} \|W^{-\frac{p'}{p}}\|_{\text{A}_{p'}}  \approx \|W\|_{\text{A}_p} ^\frac{p'}{p}. \end{equation*}

Thus, the conclusion of Lemma \ref{MatrixFKSLem} applies to the exponent $p'$ with respect to the maximal function \begin{equation*} M_{\W{W}, 1} ' \V{g} (x) = \sup_{\substack{P \ni x \\ P \in \D}} \frac{1}{|P| ^{1 - \frac{1}{d}}} \int_{P} |(V_P ' )^{-1} W^{\frac{1}{p}} (y) \V{g}(y)| \, dy. \end{equation*}  \noindent We now applying Lemma \ref{StoppingTimeLemma} again with respect to the exponent $p'$ and the matrix A${}_{p'}$ weight $\W{W}$, and repeating word for word the arguments above (interchanging the roles of $\vec{g}$ and $F$).  More precisely,  pick $C_1, C_2 > 0$ small enough (independent of $W$) so that \begin{equation*} k = \frac{p' + \frac{C_1}{\Apprime{\W{W}}}}{p' - \frac{C_2}{\Apprime{\W{W}}^\frac{p}{p'}}} = \frac{p' + \frac{C_1}{\Ap{W} ^\frac{p'}{p}}}{p' - \frac{C_2}{\Ap{W}}} < \frac{d}{d-1} \end{equation*} and \begin{equation*}q = p' - \frac{2 C_2}{\Apprime{\W{W}}^\frac{p}{p'}} = p' - \frac{2 C_2}{\Ap{W}} , \   q^* = p' - \frac{ C_2}{\Ap{W}} \end{equation*} is close enough to $p'$ and as before $\W{q} = {\color{red}p'} + \epsilon$ for $\epsilon = C_1 \Ap{W} ^{-\frac{p'}{p}}$.  Then  we get that \begin{align*} \left|\langle W^\frac{1}{p} \V{f}, \V{g}\rangle_{L^2} \right| & \lesssim a \Ap{W} ^\frac{1}{p}\epsilon^{-\frac{1}{p}} \| W^\frac{1}{p}  D\V{f}     \|_{L^{\tilde{q}'}(Q)} \|M_{\W{W}, 1} \V{g}\|_{L^{\tilde{q}}(Q)} \\  & = a \Ap{W} ^\frac{1}{p} \epsilon^{-\frac{1}{p}}  |Q|^\frac{1}{\tilde{q}}  \| W^\frac{1}{p}  D\V{f}     \|_{L^{\tilde{q}'}(Q)} \|M_{\W{W}, 1} \V{g}\|_{L^{\tilde{q}}(Q), \frac{dx}{|Q|}} \\ & \lesssim a \Ap{W} ^\frac{1}{p} \epsilon^{-\frac{1}{p}}  [W^{-\frac{p'}{p}}]_{\text{A}_{p'}} ^{\frac{p}{(p')^2}} |Q|^{\frac{1}{d} + \frac{1}{\tilde{q}} - \frac{1}{q^*}} \|W^\frac{1}{p}  D\V{f}     \|_{L^{\tilde{q} '} (Q)} \|\V{g}\|_{L^{q^*}(Q) } \\ & \lesssim   \Ap{W} ^{2 + \frac{p'}{p}} |Q|^{\frac{1}{d} + \frac{1}{\tilde{q}} - \frac{1}{q^*}} \|W^\frac{1}{p}  D\V{f}     \|_{L^{\tilde{q} '}(Q)} \|\V{g}\|_{L^{q^*}(Q) }\end{align*} since \begin{equation*} a \approx \Apprime{W^{-\frac{p'}{p}}} ^\frac{1 + p}{p'} = \Ap{W}^{1 + \frac{1}{p}} \text{ and } \Apprime{W^{-\frac{p'}{p}}} ^{\frac{p}{(p')^2}}  = \Ap{W}^\frac{1}{p'} \end{equation*} which completes the proof when $p > d$. \end{proof}

Finally we end this section with the proof of Theorem \ref{MainThmPoin}.

\begin{proof}[Proof of Theorem \ref{MainThmPoin}]  Without loss of generality, assume $\V{f}$ and $\V{g}$ are supported on $Q$.  Then again by standard arguments, we have for $x \in Q$ that \begin{equation} \label{Rep} f_i (x) - (f_i)_Q = -\frac{1}{|Q|} \int_Q \int_0^{|x - y|} \frac{\left\langle \nabla f_i (x + \frac{r (y - x)}{|y - x|}), (x - y) \right\rangle_{\Rd} }{|x - y|^{\color{red} \cancel{d}}} \, dr \, dy \end{equation}  so that
\begin{align*} \lefteqn{|\langle W^\frac{1}{p} (x) (\V{f}(x) - \V{f}_Q)  , \V{g}(x) \rangle_{\Cn}|}\\ &  \leq  \frac{1}{|Q|} \int_Q \int_0^{|x - y|} \, \frac{|\langle W^\frac{1}{p}(x) (D\V{f}  (x + \frac{r (y - x)}{|x - y|})) (x - y), \V{g} (x) \rangle_{\Cn}|  }{|y - x|^{\color{red} \cancel{d}}} \, dr \, dy
\end{align*}

By again standard arguments (see \cites{J} p. 226 for example), we therefore have
\begin{align*} \lefteqn{ |\langle W^\frac{1}{p} (x) (\V{f}(x) - \V{f}_Q),\vec{g}(x)\rangle_{\Cn}|  \chi_Q(x)}\\
 & \lesssim \int_Q \frac{|\langle W^\frac{1}{p}(x) (D\V{f} (y)) (x - y), \V{g}(x)\rangle_{\Cn}| }{|x - y|^{d}} \, dy\\
&\lesssim \sum_t\sum_{I\in \mathscr D^t}\frac{1}{|I|}\int_I\int_I |\langle W^\frac{1}{p}(x) (D\V{f} (y)) (x - y), \V{g}(x)\rangle_{\Cn}| \end{align*} The proof is now the same as the proof of the Theorem \ref{MainThmSob} starting with inequality \eqref{SumOut}.

\end{proof}

 Note that if $\vec{f} \in C^1(B)$ for an open ball $B$ then \eqref{Rep} holds with $Q$ replaced by $B$, so in this case if $Q$ is a cube containing $B$ with comparable side length, then  \begin{align*} |\langle & W^\frac{1}{p} (x) (\V{f}(x) - \V{f}_B),\vec{g}(x)\rangle_{\Cn}|  \chi_B(x)
  \\ & \lesssim {\color{red} \chi_B(x)} \int_Q \frac{|\langle W^\frac{1}{p}(x) (\chi_B (y) D\V{f} (y)) (x - y), \V{g}(x)\rangle_{\Cn}| }{|x - y|^{d }} \, dy \end{align*} so arguing as we did before proves Theorem \ref{MainThmPoin} for open balls.

\section{Existence of degenerate elliptic systems}\label{existsec}
For our existence results we will consider general nonlinear elliptic systems whose degeneracy is governed by a matrix A$_p$ weight. {\color{red} Furthermore, as in \cite{FKS}, we will only consider the case that $\Omega$ is bounded}.  Note that we will only deal with real valued systems and solutions (at least in the nonlinear case) in order to apply the abstract results in \cite{KS}.  Consider a mapping $\A:\R^d\times \MndR \ra \MndR$ such that $x\mapsto \A(x,\eta)$ is measurable for all $\eta \in \MndR$ and $\eta\mapsto \A(x,\eta)$ is continuous for a.e. $x\in\R^d$.  Note that this makes the mapping $x\mapsto \A(x,\eta(x))$ a measurable mapping whenever $\eta(x)$ is a measurable matrix valued function.  We will assume that $1<p<\infty$ and $\A$ satisfies
\vspace{3mm}

(i) $ \langle\A(x,\eta),\eta\rangle_{\text{tr}}\gtrsim \|W^{1/p}(x)\eta\|^p,  \qquad \eta \in \MndR$

(ii) $| \langle\A(x,\eta),\nu\rangle_{\text{tr}}|\lesssim \|W^{1/p}(x)\eta\|^{p-1}\|W^{1/p} (x) \nu\|, \qquad \eta,\nu\in \MndR$

(iii) $\langle \A(x,\eta)-\A(x,\nu), \eta-\nu\rangle_{\text{tr}}\geq 0$, \qquad $\eta,\nu\in \MndR$

where $\langle \,\cdot\,,\,\cdot\,\rangle_{\text{tr}}$ is the Frobenius inner product defined by
$$\langle A,B\rangle_{\text{tr}}=\text{tr}(B^* A)=\sum_{j=1}^n\sum_{k=1}^d {A}_{jk} B_{jk}$$ and is modified accordingly when $A$ and $B$ have complex entries.

A typical example  of a non-linear operator $\A$ (and one that will be discussed more in the last section) is given by the degenerate system of $p$-Laplace operators
$$\A(x,\eta)=\langle \eta G(x),\eta\rangle_{\text{tr}}^{\frac{p-2}{2}} \eta G(x)$$ where $G: \Rd \rightarrow {\color{red}\MdR}$ and $G$ satisfies

\vspace{3mm}
(i'') $ \langle \eta G(x) ),\eta \rangle_{\text{tr}}\gtrsim \|W^{1/p}(x)\eta\|^2,  \qquad \eta \in \MndR$

(ii'') $| \langle \eta G(x) , \nu\rangle_{\text{tr}}|\lesssim \|W^{1/p}(x)\eta\|\|W^{1/p} (x) \nu\|, \qquad \eta,\nu\in \MndR$

\noindent which is automatically satisfied with $n = d$ and $W = G^\frac{p}{2}$ (see the computations on p. $375$ in \cite{IM} for a verification of (iii).)  Of course, if $G : \Rd \rightarrow \MnR$ then one could also look at $$\A(x,\eta)=\langle G(x) \eta ,\eta\rangle_{\text{tr}}^{\frac{p-2}{2}} G(x) \eta $$ \noindent and write similar obvious conditions like (i'') and (ii'') for this nonlinear operator to satisfy (i) and (ii) (see again p. 378 for the simple verification of (iii)).

Such degenerate systems arise from minimizing the energy functional
$$\mathcal E(\vec{g})=\int_\Omega \langle D\vec{g}(x)G(x),D\vec{g}(x)\rangle_{\text{tr}}^{\frac{p}{2}}\,dx.$$
These systems also arise naturally in the theory mappings of finite distortion \cite[Chapter 15]{IM}.

We will now be concerned with the following system of equations in a domain $\Omega$:
\begin{equation}\label{diverform} \Div\A(x,D\vec{u}(x))= - (\Div F)(x) , \end{equation}
where $\vec{u}:\Omega\ra \mathbb{R}^n$, $\Div F(x)=(\dv F^1(x),\ldots,\dv F^n(x)),$ and $F^i$ are the row vectors of $F(x)$.  We will focus on  weak solutions to \eqref{diverform}:
\begin{equation} \int_\Omega \langle \A(x,D\vec{u}(x)),D\vec{\varphi}(x)\rangle_{\text{tr}}\,dx= - \int_{\Omega} \iptr{F(x)}{D\V{\phi} (x)}\, dx \label{LinPDE} \end{equation} for any $\V{\varphi} \in C_c^\infty(\Omega).$
As mentioned in the {\color{red}introduction}, a natural domain for these types of systems of equations is given by the matrix weighted Sobolev space H$^{1,p}(\Omega,W)$ defined as the completion of $\{ \V{u} \in C^\infty(\Omega) : \V{u}, D \V{u} \in L^p(\Omega,W)\}$ with respect to the norm:
$$\|\vec{u}\|_{\text{H}^{1,p}(\Omega,W)}=\left(\int_\Omega |W^{\frac1p}(x)\vec{u}(x)|^p\,dx+\int_\Omega \|W^{\frac1p}(x)D\vec{u}(x)\|^p\,dx\right)^{\frac1p}.$$
Moreover, the space H$^{1,p}_0(\Omega,W)$ is the completion of $C_0^\infty(\Omega)$ in the norm $\|\cdot\|_{\text{H}^{1,p}(\Omega,W)}.$   The elements of $\text{H}^{1,p}(\Omega,W)$ (or $\text{H}_0^{1,p}(\Omega,W)$) are Cauchy sequences.  However, given any element $\{(\vec{u}_k,D\vec{u}_k)\}\in \text{H}^{1,p}(\Omega,W)$ there exists $\vec{u}$ and matrix $U$ such that $\vec{u}_k\ra \vec{u}$ in $L^p(\Omega,W)$ and $D\vec{u}_k\ra U$ in $L^p(\Omega,W)$.  The elements are unique provided that $\|W\|, {\color{red}\|W^{-\frac{1}{p}}\|^{{p'}}}$ are locally integrable (see \cite[p.14]{HKM}){\color{red}, and furthermore, $U = D\V{u}$ in the distributional sense.}  If $W\in A_p$ then the exact same arguments that are contained in \cite[Theorem 5.3]{CMR}, prove that \begin{equation*} {\text H}^{1, p}(\Omega, W) = \{\vec{u}\in \mathcal{W}^ {1, 1} _{\textit{loc}} : \vec u, D\vec u\in L^p(\Omega, W)\}.\end{equation*}
{\color{red}\sout{i.e., this shows that $U=D\vec{u}=\lim_k D\vec{u_k}$}}  Hereafter, when working with the solutions \eqref{diverform} we will always assume that $W\in A_p$.

As is customary we will only prove the existence of weak solutions when $F = 0$. Also note that, with the exception of Theorem \ref{nonlinexist} below, vector functions $\V{u}$ in H$^{1,p}(\Omega,W)$ and H$^{1,p}_0(\Omega,W)$ will assumed to be $\Cn$ valued rather than $\R^n$ valued.

\begin{theorem} \label{nonlinexist} Suppose that $W\in$ A$_p$ and $\A$ satisfies (i), (ii), and (iii) above.  If $\vec h\in \text{H}^{1,p}(\Omega,W)$, then the system
$$\Div \A(x,D\vec u)=0$$
has a weak solution such that $\vec u-\vec h\in \text{H}_0^{1,p}(\Omega,W)$.
\end{theorem}
We will follow Kinderlehrer and Stampacchia \cite{KS}.  Let $X$ be a reflexive Banach space will dual space $X^*$. If ${K}$ is a convex subset of $X$ then a mapping $\mathfrak A:{K}\ra X'$ is said to be {\it monotone} if
$$\langle \mathfrak A u-\mathfrak A v,u-v\rangle \geq 0, \qquad u,v\in { K}$$
and is {\it coercive} on { K} if there exists $\varphi\in { K}$ such that
\begin{equation}\label{coercive}\frac{\langle \mathfrak A u_j-\mathfrak A \varphi,u_j-\varphi\rangle}{\|u_j-\varphi\|}\ra \infty\end{equation}
whenever $\{u_j\}$ is a sequence in ${K}$ with $\|u_j\|\ra \infty$.  The following proposition is in Kinderlehrer and Stampacchia \cite[p. 87]{KS}.

\begin{proposition} \label{existprop}Let $K\not=\varnothing$ be a closed convex subset of a reflexive Banach space $X$ and $\mathfrak A:K\ra X^*$ be monotone, coercive, and weakly continuous on $K$.  Then there exists $u\in K$ such that
$$\langle \mathfrak Au,g- u\rangle \geq 0, \qquad \forall g\in K.$$

\end{proposition}

\begin{proof}[Proof of Theorem \ref{nonlinexist}] Let $X=L^p(\Omega,\MndR,W)$ be the space of functions $F:\Omega\ra \MndR$ such that
$$\|F\|_{L^p(W)}^p=\int_\Omega \|W^{1/p}(x)F(x)\|^p\,dx<\infty,$$
with dual space $X^*=L^{p'}(\Omega,\MndR, W^{-p'/p})$ under the usual pairing
$$\langle F,G\rangle =\int_\Omega \langle F(x),G(x)\rangle_{\text{tr}}\,dx.$$
Define $$U_{\vec h}=\{\vec{u}\in  {\text H}^{1,p}(\Omega,W):\vec{u}-\vec{h}\in  {\text H}^{1,p}_0(\Omega,W)\}$$ and $$K=\{D\vec u:\vec u\in U_{\V{h}}\}.$$  Then $K$ is a nonempty convex subset of $X$.  To see that $K$ is closed suppose that $D\vec v_k\ra V$ in $X$.  Since $W\in$ A$_p$, by the Sobolev inequality we have that
$$\int_\Omega|W^\frac1p(x)(\vec v_k-\vec h)|^p\,dx\lesssim \int_\Omega \|W^\frac1p(x)(D\vec{v}_k-D\vec h)\|^p\,dx\leq C.$$
Since $U_{\vec h}$ is a closed convex subset of $\text{H}^{1,p}(\Omega,W)$ there exists a subsequence $\{\vec{v}_{k_j}\}$ and function $\vec{v}\in U_{\vec h}$ such that $\vec{v}_{k_j}\ra \vec{v}$ in $ {\text H}^{1,p}(\Omega,W)$.  In particular $V=D\vec{v}\in K$ and hence $K$ is closed.

For $F,G\in X$ define
$$\langle\mathfrak A F,G\rangle=\int_\Omega \langle \mathcal A(x,F(x)),G(x)\rangle_{\text{tr}}\,dx.$$
Notice by assumption (ii) on $\mathcal A$ we have that
$$|\langle\mathfrak A F,G\rangle|\leq \|F\|_{L^p(W)}^{p-1}\|G\|_{L^p(W)}$$
so that $\mathfrak A: X\ra X^*$.  From assumption (iii) on $\A$ we have that $\mathfrak A$ is monotone.  Thus we need to check that $\mathfrak A$ is coercive, i.e. satisfies condition \eqref{coercive}.  Suppose $U_k=D\vec u_k\in K$ satisfies $\|U_k\|_{L^p(W)}\ra \infty.$   Then, given $V=D\vec v\in K$ we have $\|U_k-V\|_{L^p(W)}\ra \infty$ as well.  Fix $V=D\vec v\in K$ and use assumption (i) on $\A$ to get
\begin{align*}
\langle \mathfrak AU_k-\mathfrak AV,U_k-V\rangle&=\int_\Omega \langle \A (x,U_k)- \A (x,V), U_k-V\rangle_{\text{tr}}\,dx\\
&=\int_\Omega  \langle \A (x,U_k),U_k\rangle_{\text{tr}}\,dx+\int_\Omega \langle \A (x,V),V\rangle_{\text{tr}}\,dx\\
&\qquad -\int_\Omega  \langle \A (x,U_k),V\rangle_{\text{tr}}\,dx-\int_\Omega  \langle \A (x,V),U_k\rangle_{\text{tr}}\,dx\\
&\geq c(\|U_k\|_{L^p(W)}^p+\|V\|_{L^p(W)}^p)\\
&\qquad -C(\|U_k\|_{L^p(W)}^{p-1}\|V\|_{L^p(W)}+\|U_k\|_{L^p(W)}\|V\|^{p-1}_{L^p(W)})\\
&\geq c2^{-p}\|U_k-V\|_{L^p(W)}^p\\
&\qquad -C2^{1-p}[\|V\|_{L^p(W)}(\|V\|_{L^p(W)}^{p-1}+\|U_k-V\|_{L^p(W)}^{p-1})\\
&\quad -C(\|V\|_{L^p(W)}^{p-1}(\|V\|_{L^p(W)}+\|U_k-V\|_{L^p(W)})
\end{align*}
which implies $$
\frac{\langle \mathfrak AU_k-\mathfrak AV,U_k-V\rangle}{\|U_k-V\|_{L^p(W)}} \rightarrow \infty$$
as $\|U_k\|_{L^p(W)}\ra \infty$, showing that $\mathfrak A$ is coercive.  Finally, the weak continuity follows from the continuity of $\eta\mapsto \mathcal A(x,\eta)$.    By Proposition \ref{existprop} there exists $U=D\vec{u}\in K$ such that
$$\langle \mathfrak AU,G-U\rangle \geq 0, \qquad \forall G\in K.$$
If $\vec\varphi\in C_c(\Omega,\R^n)$, then $\vec u-\vec \varphi$ and $\vec u+\vec \varphi$ belong to $U_{\vec{h}}$ and hence
$$\int_\Omega \langle \mathcal A(x,D\vec u),D\vec\varphi\rangle_{\text{tr}}=0.$$

\end{proof}

We now consider the case when $\A$ is linear, that is,
$$\A(x,\eta)= \sum_{j = 1}^n \sum_{\beta = 1}^d  A_{ij}^{\alpha\beta}(x)\eta_{j\beta}.$$
In this case we will consider the nonhomgeneous system $$\Div \A(x,D\vec{u}(x))=-\Div(F(x)),$$ which clearly reduces to \eqref{linearsys}.

Moreover, when $p=2$ conditions (i) and (ii) on $\A$ simply become \eqref{MIEllip1} and \eqref{MIEllip2}, respectively.  Moreover, condition (iii) on $\A$ is automatically satisfied by the linearity of $\A$.  In this case we have an existence and uniqueness result, which follows from a standard use of the Lax-Milgram theorem (where here and in the rest of the paper we assume the entries of $A, \V{u}, F$ and $\V{h}$ are complex.)  .

\begin{theorem} \label{WeakExistence} Let $A$ satisfy \eqref{MIEllip1} and \eqref{MIEllip2}, $\vec{h} \in H^{1,2} (\Omega, W)$, and $ F \in L^2(\Omega, W^{-1})$. Then the system \eqref{linearsys} has a unique weak solution $\vec{u} \in H^{1, 2}(\Omega, W)$ such that $\V{u} - \V{h} \in H_0 ^{1,2} (\Omega, W)$ \end{theorem}

\section{Basic regularity results}\label{regularity}

We now discuss some deeper results (that still are fairly elementary from the elliptic PDE point of view.)  The first is a degenerate Caccioppoli inequality.

\begin{lemma} Assume that $\Omega$ is some open set and $B_r$ is any ball of radius $r$ whose closure is contained in $\Omega$.  If $A = A_{ij}^{\alpha \beta}$ satisfies \eqref{MIEllip1} and \eqref{MIEllip2}, and if  $\vec{u} \in H^{1,2}(\Omega, W)$ is a solution to \eqref{linearsys} for $F \in L^2(\Omega, W^{-1})$, then \begin{align} \int_{B_{r/2}} \|W^\frac{1}{2} (x) D\vec{u} (x) \|^2 \, dx & \lesssim \frac{1}{r^2} \int_{B_r \backslash B_{r\backslash 2}} |W^\frac{1}{2} (x) \vec{u} (x)| ^2 \, dx \nonumber \\ & + \int_{B_r} \|W^{-\frac{1}{2}} (x) F(x) \|^2 \, dx \label{CaccioIneq}  \end{align} \end{lemma}

\begin{remark} We do not need to assume any conditions on our matrix weight $W$ other than positive definiteness a.e.  In particular, the constants in our Cacciopoli inequality do not depend on the A${}_2$ characteristic.
\end{remark}

\begin{proof} The proof is classical, and the only nontrivial thing to check is that our system and degeneracy is ``decoupled" enough. Pick some $\eta \in C_c^\infty(B_r)$ such that  $\eta \equiv 1$ in $B_{r/2}, \ 0 \leq \eta \leq 1$ in $B_r$, and $|\nabla \eta|  \leq \frac{4}{r} \chi_{B_r \backslash  B_{r/2}} $ and let $\V{\varphi} := \eta^2\V{u} \in H_0 ^{1, 2} (\Omega, W)$.  By definition, we have that \begin{equation} \sum_{\alpha, \beta, i, j} \int_{B_r} A_{ij}^{\alpha \beta} (\partial_\beta u_j ) \overline{(\partial_\alpha (u_i \eta^2) )} \, dx = - \int_{B_r} \iptr{F}{D(\eta^2 \V{u})} \, dx. \label{PDESol}\end{equation} However,  \begin{equation*} \partial_\alpha (u_i \eta^2)  = u_i (2 \eta \partial_\alpha \eta) + \eta^2 \partial_\alpha u_i  \end{equation*} so that \begin{equation} D(\eta^2 \V{u}) = 2 (\eta \V{u}) \otimes \nabla \eta + \eta^2 D\V{u} \label{ProdRule} \end{equation} so combining this with \eqref{MIEllip1}, \eqref{MIEllip2},  and \eqref{PDESol} gives \begin{align*}  \int_{B_{r}} & |\eta  |^2 \|W^\frac{1}{2}    D\V{u}  \| ^2 \, dx  \\ & \lesssim \sum_{\alpha, \beta, i, j} \int_{B_r} A_{ij}^{\alpha \beta}  (\partial_\beta u_j)  (\overline{\eta^2 \partial_\alpha u_i  }) \, dx \\ &  \leq \left|\sum_{\alpha, \beta, i, j} \int_{B_r } A_{ij}^{\alpha \beta}  \partial_\beta u_j  \overline{u_i (2 \eta \partial_\alpha \eta) } {\color{red} \, dx }\right|   + \int_{B_r} \|W^{-\frac{1}{2}} F \| \|W^\frac{1}{2} D(\eta^2 \V{u}) \| \, dx \\ & = 2\left|\sum_{\alpha, \beta, i, j} \int_{B_r }  A_{ij}^{\alpha \beta}  (\partial_\beta u_j)  (\overline{(\eta\V{u} \otimes \nabla \eta)_{i \alpha} }) {\color{red} \, dx } \right| + \int_{B_r} \|W^{-\frac{1}{2}} F \| \|W^\frac{1}{2} D(\eta^2 \V{u}) \| \, dx \\ & \lesssim  \int_{B_r }  |\eta| \|W^\frac12  D \V{u} \| \|W^\frac12  (\V{u} \otimes \nabla \eta) \| \, dx  + \int_{B_r} \|W^{-\frac{1}{2}} F \| \|W^\frac{1}{2} D(\eta^2 \V{u})\| \, dx \\ & \leq \int_{B_r }  |\eta| |\nabla \eta| \|W^\frac12 D \V{u} \| |W^\frac12  \V{u} | \, dx  + \int_{B_r} \|W^{-\frac{1}{2}} F \| \|W^\frac{1}{2} D(\eta^2 \V{u})\| \, dx.
\end{align*}

Thus, by the``Cauchy-Schwarz inequality with $\epsilon$ " we have

\begin{align*} \int_{B_{r}}  |\eta  |^2 \|W^\frac{1}{2}    D\V{u}  \| ^2 \, dx  & \lesssim \epsilon \int_{B_r }   |\eta|^2 \|W^\frac12 D \V{u} \|^2  \, dx + \frac{C(\epsilon)}{r^2} \int_{B_r \backslash B_{r/2}}   |W^\frac12  \V{u} |^2  \, dx \\ &  + \epsilon \int_{B_r}  \|W^\frac{1}{2} D(\eta^2 \V{u})\|^2 \, dx + C(\epsilon) \int_{B_r} \|W^{-\frac{1}{2}} F \| ^2 \, dx \end{align*} for some $C(\epsilon) > 0$.

  However, \eqref{ProdRule}  gives us that  \begin{equation*} \epsilon  \int_{B_r }  \|W^\frac{1}{2} D(\eta^2 \V{u})\|^2 \, dx \lesssim \frac{\epsilon}{r^2} \int_{B_r \backslash B_{r/2}} |W^\frac12 \V{u}|^2 \, dx + \epsilon \int_{B_r} | \eta|^2 \|W^\frac12 D\V{u}\|^2 \, dx \end{equation*}  so finally

\begin{align*}  \int_{B_{r}}  & |\eta  |^2 \|W^\frac{1}{2}    D\V{u}  \| ^2 \, dx \\ & \lesssim \epsilon \int_{B_r }   |\eta|^2 \|W^\frac12 D \V{u} \|^2  \, dx   + \frac{C(\epsilon)}{r^2} \int_{B_r \backslash B_{r/2}}   |W^\frac12  \V{u} |^2  \, dx + C(\epsilon) \int_{B_r} \|W^{-\frac{1}{2}} F \| ^2 \, dx.\end{align*} Setting $\epsilon > 0$ small enough and remembering that $\eta \equiv 1$ on $B_{r/2}$ finishes the proof.
\end{proof}

We now prove Theorem \ref{revmeyer} as a Corollary of our Caccioppoli inequality.

\noindent \textit{Proof of Theorem} \ref{revmeyer}.   Let $\epsilon > 0$ be chosen where Theorem \ref{MainThmPoin} is true, so by \eqref{CaccioIneq} \begin{align*} & \left( \frac{1}{|B_{r/2}| } \int_{B_{r/2}} \|W^{\frac{1}{2}}  D \V{u}  \|^2 \, dx\right)^\frac{1}{2}   \\ & \lesssim \left(\frac{1}{|B_r|} \int_{B_r} \|W^{-\frac{1}{2}}  F \|^{2} \, dx \right)^\frac{1}{2}     + \frac{1}{r} \left(\frac{1}{|B_r|} \int_{B_r} |W^{\frac{1}{2}}  (\V{u} - \V{u}_{B_r})|^2 \, dx \right)^\frac{1}{2}  \\ & \lesssim \left(\frac{1}{|B_r|} \int_{B_r} \|W^{-\frac{1}{2}}  F \|^{2} \, dx \right)^\frac{1}{2}  + \left(\frac{1}{|B_r|} \int_{B_r} \|W^\frac{1}{2}  D\V{u} \|^{2 - \epsilon} \, dx \right)^\frac{1}{2 - \epsilon} \end{align*}  However, setting \begin{equation*} U(x) = \|W^\frac{1}{2} (x)  D\V{u} (x)\|^{2 - \epsilon}, \ \ \  G(x) = \|W^{-\frac{1}{2}} (x) F (x)\|^{2 - \epsilon}, \text{ and } s = \frac{2}{2 - \epsilon}  \end{equation*}  we have that  \begin{equation*} \frac{1}{|B_{r/2}|}  \int_{B_{r/2}} (U(x))^s \, dx  \lesssim \left( \frac{1}{|B_{r}|} \int_{B_{r}} U(x) \, dx \right)^s + \frac{1}{|B_r|} \int_{B_r} (G(x))^s \, dx . \end{equation*}  A classical result of Giaquinta and Modica (see  Lemma $2.2$ in \cite{FBT}) now says that there exists $t > s = \frac{2}{2 - \epsilon}$ where \begin{align*}  & \left( \frac{1}{|B_{r/2}|} \int_{B_{r/2}} \|W^\frac{1}{2} D\V{u}  \|^{t(2 - \epsilon)} \, dx \right)^\frac{1}{t} \\ & \lesssim \left( \frac{1}{|B_{r}|} \int_{B_{r}} \|W^\frac{1}{2}  D\V{u} \|^2 \, dx \right)^\frac{2- \epsilon}{2}  + \left(\frac{1}{|B_r|} \int_{B_r} \|W^{-\frac{1}{2}}  F\|^{t (2 - \epsilon)} \, dx \right)^\frac{1}{t}. \end{align*} Setting $q = t(2 - \epsilon) > 2$ clearly completes the proof. \hfill $\square$.

We can now prove a decay of solutions type theorem, where here we do assume that $W$ is a matrix A${}_2$ weight. For simplicity we will assume $F = 0$ in our linear elliptic system. First however we need the following two lemmas, which will prove a sort of ``weak" Poincar\'{e} inequality for annuli.

\begin{lemma} \label{lem:constosc}For any $\vec{a} \in \Cn$ , a matrix A${}_p$ weight $W$, and $B = B_{r/2} \backslash B_{r/4}$ where $B_{r/2}$ and $B_{r/4}$ are concentric balls of radius $r/2$ and $r/4$, respectively, we have \begin{multline*} \left(\frac{1}{|B|} \int_B |W^\frac{1}{p}(x) ( \V{u}(x) - \V{u}_{B})| ^p \, dx \right)^\frac{1}{p} \\ \lesssim  [W]_{\text{A}_p} ^\frac{1}{p} \left(\frac{1}{|B|} \int_B |W^\frac{1}{p}(x)  (\V{u}(x) - \V{a})| ^p \, dx \right)^\frac{1}{p}. \end{multline*} \end{lemma}

\begin{remark} As will be apparent from the proof, one can state and prove a similar result for sets $B$ that aren't necessarily annuli as above.  We will leave this for the interested reader to do this. \end{remark}

\begin{proof}
 By the triangle inequality, \begin{align*} \left(\frac{1}{|B|} \int_B |W^\frac{1}{p}(x)  (\V{u}(x) - \V{u}_{B} )| ^p \, dx \right)^\frac{1}{p} & \leq \left(\frac{1}{|B|} \int_B |W^\frac{1}{p}(x) ( \V{u}(x) -  \V{a})| ^p \, dx \right)^\frac{1}{p} \\ &  + \left(\frac{1}{|B|} \int_B |W^\frac{1}{p}(x) ( \V{a} - \V{u}_{B})| ^p \, dx \right)^\frac{1}{p} \end{align*}  However,  \begin{align*} |W^\frac{1}{p}(x) ( \V{a} -\V{u}_{B})|^p & = \left| \frac{1}{|B|} \int_B W^\frac{1}{p}(x) (\V{u}(y) - \V{a} ) \, dy \right|^p  \\ & = \left| \frac{1}{|B|} \int_B (W^\frac{1}{p}(x) W^{- \frac{1}{p}}(y)) W^\frac{1}{p}(y) (\V{u}(y) - \V{a} ) \, dy \right|^p \\ & \leq \left( \frac{1}{|B|} \int_B \|W^\frac{1}{p}(x) W^{- \frac{1}{p}}(y)\|^{p'} \, dy \right)^\frac{p}{p'}\\
 &\qquad\qquad\times \left(\frac{1}{|B|} \int_B |   W^\frac{1}{p}(y) (\V{u}(y) - \V{a} )| ^p  \, dy \right)\end{align*}
 Plugging this in and using the A${}_p$ definition immediately finishes the proof.

\end{proof}

\begin{lemma} \label{PoinAnnuli}

Let $W$ be a matrix A${}_p$ weight and assume that $\vec{u} \in H^{1, 2}(\Omega, W)$ for some open set $\Omega$.  If $B_{r} \subseteq \Omega$ and $C$ is the constant from Theorem \ref{MainThmPoin},  then \begin{equation*} \int_{B_{r/2} \backslash B_{r/4}} |W^\frac{1}{p}(x) (\vec{u} (x) - \V{u}_{B_{r/2} \backslash B_{r/4}}) |^p \, dx \lesssim C^2  \Ap{W} ^{2} r^p \int_{B_{r} \backslash B_{r/8}} {\color{red} \|} W^{\frac{1}{p}} (x) D\V{u}(x)  {\color{red} \|}^p \, dx \end{equation*} where $B_r, B_{r/2}, B_{r/4}$ and $B_{r/8}$ are concentric balls. \end{lemma}

\begin{proof} The proof utilizes standard geometric ideas.  Let $\{x_j\}_{j = 1}^N \subseteq B_{r/2} \backslash B_{r/4} $ be a maximal set satisfying \begin{equation*} \min_{i \neq j} |x_i - x_j| \geq r/16. \end{equation*}  The  balls $\{B_{r/16} (x_j)\}_{j = 1}^N $ cover $B_{r/2} \backslash B_{r/4}$ and a trivial volume-count gives us that we can find an upper bound for $N$ depending only on $d$.  Finally, by introducing repeats if necessary, we can without loss of generality assume that \begin{equation*} B_{r/16} (x_j) \cap B_{r/16} (x_{j+1}) \neq \emptyset \text{ for all } j = 1, \ldots, N - 1 \end{equation*} so that for each $ j = 1, \ldots, N - 1$ there exists $v_j$ where \begin{equation*} B_{r/16}(v_j)  \subseteq B_{r/8} (x_j) \cap B_{r/8} (x_{j + 1}). \end{equation*}

   For notational ease, let $\tilde{B}_j = B_{r/8} (x_j) \cap B_{r/8} (x_{j + 1})$.  Clearly by the previous lemma it is enough to prove that \begin{equation*} \int_{B_{r/2} \backslash B_{r/4}} |W^\frac{1}{p}(x) (\vec{u} (x) - {\V{u}}_{B_{r/8}(x_1)}  )|^p \, dx \lesssim \Ap{W} r^p  \int_{B_{r} \backslash B_{r/8}} |W^{\frac{1}{p}} (x) D\V{u}(x)  |^p \, dx. \end{equation*}

   \noindent To that end, we have \begin{align*} \int_{B_{r/2} \backslash B_{r/4} }  & |W^{\frac{1}{p}} (x) (\vec{u} (x) - \V{u}_{B_{r/8}(x_1)} ) |^p \, dx  \\ & \leq \sum_{j = 1}^N \int_{B_{r/16} (x_j)} |W^\frac{1}{p} (x) (\vec{u} (x) - \V{u}_{B_{r/8}(x_1)}) |^p \, dx \\ & {\color{red} \lesssim} \sum_{j = 1}^N \int_{B_{r/8} (x_j)} |W^\frac{1}{p}(x) (\vec{u} (x) - \V{u}_{B_{r/8}(x_j)}) |^p \, dx  \\ & \ \ \ + \sum_{j = 1}^N \int_{B_{r/8} (x_j)} |W^\frac{1}{p}(x) (\V{u}_{B_{r/8}(x_j)} - \V{u}_{B_{r/8}(x_1)}) |^p \, dx  \\ & \lesssim C r^p \int_{B_{r} \backslash B_{r/8}} {\color{red} \|} W^{\frac{1}{p}} (x) D \V{u}(x) {\color{red}  \|}^p \, dx   \\ & \ \ \ + \sum_{j = 1}^N \int_{B_{r/8} (x_j)} |W^\frac{1}{p}(x) (\V{u}_{B_{r/8}(x_j)} - \V{u}_{B_{r/8}(x_1)}) |^p \, dx. \end{align*}

However, \begin{align*} | W^\frac{1}{p}(x) & (\V{u}_{B_{r/8}(x_j)} - \V{u}_{B_{r/8}(x_1)}) |^p  \\ & \lesssim  \sum_{i = 1} ^{j - 1} |W^\frac{1}{p}(x) (\V{u}_{B_{r/8}(x_{i + 1})} - \V{u}_{B_{r/8}(x_{i })}|^p \\ & \lesssim \sum_{i = 1} ^{j - 1} \left(|W^\frac{1}{p}(x) (\V{u}_{B_{r/8}(x_{i + 1})} - \V{u}_{\tilde{B}_i})|^p +  |W^\frac{1}{p}(x) (\V{u}_{\tilde{B}_i} - \V{u}_{B_{r/8}(x_{i })})|^p \right).\end{align*} \noindent Moreover, \begin{align*} |W^\frac{1}{p}(x) & (\V{u}_{B_{r/8}(x_{i + 1})} - \V{u}_{\tilde{B}_i})|^p  \\ & \leq  \left(\frac{1}{|\tilde{B}_i|} \int_{\tilde{B}_i} | W^\frac{1}{p} (x) (\V{u}(y) - \V{u}_{B_{r/8}(x_{i + 1})}) | \, dy \right)^p \\ & \lesssim \left(\frac{1}{|B_{r/8}(x_{i + 1})|} \int_{B_{r/8}(x_{i + 1})} \| W^\frac{1}{p} (x) W^{-\frac{1}{p}} (y)\| ^{p'} \, dy \right)^\frac{p}{p'}  \\ & \times \frac{1}{|B_{r/8}(x_{i + 1})|} \int_{B_{r/8}(x_{i + 1})} |W^\frac{1}{p}(y) ( \V{u}(y) - \V{u}_{B_{r/8}(x_{i + 1})}) |^p \, dy.  \end{align*} and a similar estimate holds for $|W^\frac{1}{p}(x)  (\V{u}_{B_{r/8}(x_{i })} - \V{u}_{\tilde{B}_i})|^p $.

Thus, by the matrix A${}_p$ property \begin{equation*} \sum_{j = 1}^N \int_{B_{r/8} (x_j)} |W^\frac{1}{p}(x) (\V{u}_{B_{r/8}(x_j)} - \V{u}_{B_{r/8}}(x_1)) |^p \, dx  \lesssim C \Ap{W}  r^p  \int_{B_{r} \backslash B_{r/8}} |W^{\frac{1}{p}} (x) D\V{u}(x)  |^p \, dx \end{equation*} which completes the proof.
\end{proof}

\begin{theorem} Assume $A = A_{ij}^{\alpha \beta}$ satisfies \eqref{MIEllip1} and  \eqref{MIEllip2} for some $W\in$ A$_2$ and that $\vec{u} \in H^{1,2}(\Omega, W)$ is a weak solution to \eqref{linearsys} for $F = 0$.  Then there exists $\gamma \approx [W]_{\text{A}_2} ^{-8}$  where \begin{equation*}   \int_{B_{r} } \|W^\frac{1}{2} (x) D\vec{u} (x) \|^2 \, dx \lesssim  \left(\frac{r}{R}\right)^\gamma \int_{B_R } \|W^\frac{1}{2} (x) D \vec{u} (x) \| ^2 \, dx. \end{equation*} for every concentric ball $B_r \subset B_R $ with the closure of $B_R$ contained in $\Omega$.    \end{theorem}

\begin{proof} The proof involves a  ``Widman hole filling" argument.   Note that if $\V{u}$ is a weak solution then obviously $  \V{u} - \V{u}_{B_{r/2} \backslash B_{r/4}  }$ is also a weak solution.  Thus, by  Lemma \ref{PoinAnnuli} and \eqref{CaccioIneq}, we can pick $C > 0$ independent of $r$ and $W$ where \begin{align*}   \int_{B_{r/4}} \|W^\frac{1}{2} (x) D\vec{u} (x) \|^2 \, dx & \leq \frac{C}{r^2} \int_{B_{r/2} \backslash B_{r/4}} |W^\frac{1}{2} (x) (\vec{u} -   \V{u}_{B_{r/2} \backslash B_{r/4}  })| ^2 \, dx \\ & \leq C [W]_{\text{A}_2} ^{8} \int_{B_r \backslash B_{r/ 8}} \|W^\frac{1}{2} (x) D\V{u} (x)\| ^2 \, dx \end{align*} which means

\begin{equation*} (C [W]_{\text{A}_2} ^{8} + 1) \int_{B_{r/8}} \|W^\frac{1}{2} (x) D\vec{u} (x) \|^2 \, dx \leq C [W]_{\text{A}_2} ^{8}  \int_{B_r  } \|W^\frac{1}{2} (x) D\V{u} (x)| ^2 \, dx \end{equation*} or \begin{equation*}  \int_{B_{r/8}} \|W^\frac{1}{2} (x) D\vec{u} (x) \|^2 \, dx \leq \delta   \int_{B_r } \|W^\frac{1}{2} (x) D\V{u} (x)\| ^2 \, dx \end{equation*} where $\delta = \frac{C [W]_{\text{A}_2} ^{8}}{C [W]_{\text{A}_2} ^{8} + 1}$.

Finally, if $2^{- 3k  - 3} R < r \leq 2^{-3k} R$ and $\gamma = - \frac{\ln \delta}{3\ln 2} \approx [W]_{\text{A}_2} ^{-8} $ then \begin{equation*} \int_{B_r}  \|W^\frac{1}{2} (x) D\vec{u} (x) \|^2 \, dx \leq 8^\gamma \left(\frac{r}{R}\right)^{\gamma } \int_{B_R}  \|W^\frac{1}{2} (x) D\vec{u} (x) \|^2 \, dx. \end{equation*}

\end{proof}

We will now prove Theorem \ref{thm:locreg}.

\noindent \textit{Proof of Theorem \ref{thm:locreg}}.  The proof is a modification of some ideas in \cite{Lw}. Let $B = B(x_0, R)$ and for $x, y \in B$ let $B_x ^k$ and $B_y^k $ be balls of radius $2^{-k + 1} |x - y|$ and centered at $x$ and $y$ respectively, so that in particular $B_y ^0 \subseteq B(x, 6R) \subseteq {\color{red} B(x_0, 7R) \subseteq \Omega}$. For notational ease let $B_x = B_x ^0$ and similarly for $B_y$.  Note that for $\epsilon \approx [W]_{\text{A}_2} ^{-8}$ {\color{red} and $\gamma = 6 \epsilon$}  we  have by our Poincare inequality for $d = 2$ and decay of solutions lemma that
\begin{align*} \lefteqn{ \frac{1}{|B_x  ^k|^{1 +\epsilon}} \int_{B_x ^k} |\V{u} (x) - \V{u}_{B_x ^k} | \, dx }\\ & \leq \left(\frac{1}{|{B_x ^k}|^{1 - \epsilon}} \int_{B_x ^k} \|W ^{-\frac12} (x)\|^2 \, dx \right)^\frac12  \left(\frac{1}{|{B_x ^k}|^{1 + 3 \epsilon}} \int_{B_x ^k} |W^\frac12 (x) ( \V{u}(x) - \V{u}_{B_x ^k} )|^2 \, dx \right)^\frac12  \\ & \lesssim R^{-3\epsilon} \left(\frac{1}{|{B_x ^k}|^{1 - \epsilon}}  \int_{B_x ^k}  \|W^{-\frac12} (x)\|^2 \, dx \right)^\frac12  \end{align*} (since $B(x, R) \subseteq \Omega$) and obviously the same estimate holds for $B_y ^k$.

  Since $\V{u}$ is locally integrable, let \begin{equation*} \V{\mathcal{U}} (x) = \lim_{k \rightarrow \infty} \frac{1}{| B_x ^k| } \int_{B_x ^k} \V{u}(s) \, ds. \end{equation*}  Then note that by the Lebesgue differentiation theorem we have $\V{\mathcal{U}}$ coincides with $\V{u}$ a.e. and

\begin{align*} |\V{\mathcal{U}} (x) - \V{u}_{B_x}| & \leq \sum_{k = 0}^\infty |\V{u}_{B_x ^{k + 1}} - \V{u}_{B_x ^k}| \\ & \lesssim  \sum_{k = 0}^\infty \frac{1}{|B_x ^{k }|} \int_{B_{x} ^{k }} |\V{u} (s) - \V{u}_{B_x ^{k }} | \, ds \\ & \lesssim R^{-3\epsilon} C_{x, y} \sum_{k = 0}^\infty |B_x ^{k }|^\epsilon \\ & \lesssim R^{-3\epsilon} C_{x, y} |x - y|^\epsilon  \end{align*}  The estimate for $|\V{\mathcal{U}} (y) - \V{u}_{B_y}|$ is similar and finally
\begin{align*} |\V{u}_{B_y} - \V{u}_{B_x}| & = \frac{1}{|B_x ^{1}|} \int_{B_x ^1}  |\V{u}_{B_y} - \V{u}_{B_x}| \, ds
\\ & \leq \frac{1}{|B_x ^{1}|} \int_{B_x ^1}  |\V{u}_{B_y} - \V{u} (s)| \, ds  + \frac{1}{|B_x ^{1}|} \int_{B_x ^1}  |\V{u} (s) - \V{u}_{B_x}| \, ds
\\ & \lesssim \frac{1}{|B_y|} \int_{B_y}  |\V{u}_{B_y} - \V{u} (s)| \, ds  + \frac{1}{|B_x |} \int_{B_x }  |\V{u} (s) - \V{u}_{B_x}| \, ds
\\ & \lesssim R^{-3\epsilon} C_{x, y} |x - y|^\epsilon  \end{align*} since $B_x ^1 \subseteq B_y$.

 \hfill $\square$.

\begin{comment}
First if $w$ is any scalar weight and $1 < p < \infty$ then by H\"{o}lder's inequality we have \begin{equation*} (w(B))^{-1} \leq (w^{-1} (B))^\frac{p'}{p} |B|^{-p'} \end{equation*} for any Borel $|B| > 0$ where $w(B) = \int_B w$.  In particular, if $w = \|W\|$ then $w^{-1}$ is assumed to be locally integral so we get $(\|W\|(B))^{-1} \leq |B|^{-p'}.$  Thus, by the equivalence between the trace norm and operator norm, we get
\end{comment}

To prove Theorem \ref{FBTThm} we will first need to prove the following Caccioppoli inequality, which is a matrically degenerate version of the Caccioppoli inequality proved in the very recent paper \cite{FBT} for  uniformly elliptic $p$-Laplacian systems. Note that again for this Caccioppoli inequality we do not require that $W$ is a matrix A${}_p$ weight.

  \begin{lemma} \label{FBTLem} Let $p > 2$, and let $W$ and $G$ satisfy conditions $(i')$ and $(ii')$ of Theorem \ref{FBTThm}.  If $\V{u} \in H ^{1, p} (\Omega, W)$ is a weak solution to \eqref{PLap} where $F \in L^{p'}(\Omega, W^{-\frac{p'}{p}})$,   then for any ball $B_r$ whose closure is contained in $\Omega$ we have \begin{align} \int_{B_{r/2}} \|W^{\frac{1}{p}} (x) D \V{u} (x) \|^p \, dx & \lesssim \int_{B_r} \|W^{-\frac{1}{p}} (x) F (x) \|^{p'} \, dx \nonumber \\ & + \frac{1}{r^p} \int_{B_r \backslash B_{r/2} } |W^\frac{1}{p} (x) \V{u} (x) |^p \, dx.  \label{NLCaccio} \end{align} \end{lemma}

  \begin{proof}  The proof is similar to the arguments in \cite{FBT}, p. 57 - 62.  As in the proof of \eqref{CaccioIneq}, pick some $\eta \in C_c^\infty(B_r)$ where $\eta \equiv 1$ on $B_{\frac{r}{2}}, 0 \leq \eta \leq 1$ on $B_r$, and \begin{equation*} |\nabla \eta| \leq \frac{4}{r} \chi_{B_r \backslash B_{\frac{r}{2}}} \end{equation*}

Since  $\V{u} \in {\text H}^{1, p} (\Omega, W)$ is a weak solution to \eqref{PLap}  we have that \begin{equation*} \int_\Omega \left[ \ip{D\V{u} G}{D\V{u}}_{\tr}\right]^{\frac{p - 2}{2}}  \iptr{D\V{u} G}{ D(\eta^p \V{u})} \, dx = - \int_{\Omega} \iptr{F}{D(\eta^p \V{u})} \, dx. \end{equation*}

Using the equality \begin{equation*} D(\eta ^p \V{u}) =  (p - 1) \eta ^{p-2} ({(\color{red} \eta \V{u})} \otimes \nabla \eta) + \eta^{p - 1} D(\eta \V{u}) \end{equation*}  it follows that

\begin{align*}
      \iptr{ D\V{u} G}{D(\eta^p \V{u})} & =  \eta^{p -2} [ \iptr{D(\eta \V{u}) G}{D(\eta \V{u})} - \iptr{(\V{u} \otimes \nabla \eta) G}{D(\eta \V{u})}{\color{red}\cancel{]}}  \\ &  \ \ \ +  {\color{red}\cancel{[}}(p - 1) \iptr{D(\eta \V{u}) G}{\V{u} \otimes \nabla \eta} - (p - 1) \iptr{(\V{u} \otimes \nabla \eta)G }{\V{u} \otimes \nabla \eta} ] \\ & := \eta^{p - 2} \A(x, \V{u}, \eta).
    \end{align*}

\noindent  Similarly
\begin{align*} \left[\iptr{D\V{u} G}{D\V{u}} \eta^2\right]^\frac{p-2}{2}  & =  \left[\iptr{D(\eta \V{u}) G}{D(\eta \V{u})} - \iptr{D(\eta \V{u} ) G}{ \V{u} \otimes \nabla \eta}\right. \\ & \ \ \ - \left.\iptr{(\V{u} \otimes \nabla \eta)G}{D(\eta \V{u})} + \iptr{(\V{u} \otimes \nabla \eta) G}{\V{u} \otimes \nabla \eta } \right] ^{\frac{p-2}{2}}  \\ & := [\B(x, \V{u}, \eta)]^{\frac{p-2}{2}} \end{align*}

\noindent so that \begin{equation*} \int_{\Omega}  \A(x, \V{u}, \eta) [\B(x, \V{u}, \eta)]^{\frac{p-2}{2}} \, dx = {\color{red} - \int_{\Omega} \iptr{F}{D(\eta^p \V{u})} \, dx}. \end{equation*}

\noindent Furthermore, define \begin{equation*} \N(\V{u}, \eta) := \int_\Omega \left[\iptr{D(\eta \V{u}) G}{D(\eta \V{u})}\right]^{\frac{p - 2}{2}} \iptr{D(\eta \V{u}) G}{D(\eta \V{u})} \, dx \end{equation*} so by condition {\color{red}(i')} \begin{equation} \N(\V{u}, \eta) \gtrsim \int_\Omega \|W^\frac{1}{p}(x) D(\eta \V{u}) \|^p \, dx = \int_{B_r} \|W^\frac{1}{p}(x) D(\eta \V{u}) \|^p \, dx. \label{NLCaccio1} \end{equation}   We will now obtain a suitable upper bound for $|\N(\V{u}, \eta)|$.  By the definitions of $\A(x, \V{u}, \eta)$ and $\B(x, \V{u}, \eta)$ we can {\color{red}estimate}

\begin{align*} \N(\V{u}, \eta) & {\color{red}\lesssim}  \int_{\Omega} [\B(x, \V{u}, \eta)]^{\frac{p-2}{2}}\iptr{D(\eta \V{u}) G}{D(\eta \V{u})} \, dx  \\ & \ \ \ + \int_\Omega {\color{red}\big|}\iptr{D(\eta \V{u}) G}{\V{u} \otimes \nabla \eta} + \iptr{(\V{u} \otimes \nabla \eta) G}{D(\eta \V{u})} \\ & \ \ \ - \iptr{(\V{u} \otimes \nabla \eta)G }{ \V{u} \otimes \nabla \eta} {\color{red}\big|}^\frac{p-2}{2} \iptr{D(\eta \V{u}) G}{D(\eta \V{u})}\, dx \\ & =
\int_\Omega [\B(x, \V{u}, \eta)]^\frac{p - 2}{2} \A(x, \V{u}, \eta) \, dx + \int_{\Omega} [\B(x, \V{u}, \eta)]^\frac{p-2}{2} \\ & \ \ \ \times {\color{red}\bigg[} \iptr{(\V{u} \otimes \nabla \eta) G}{D(\eta \V{u})} - (p - 1) \iptr{D(\eta {\color{red}\V{u}}) G}{\V{u} \otimes \nabla \eta} \\ & \ \ \ +  (p - 1) \iptr{(\V{u}\otimes \nabla \eta) G}{\V{u} \otimes \nabla \eta} {\color{red}\bigg]}  \, dx \\ & \ \ \ + \int_{\Omega} {\color{red}\big|}\iptr{D(\eta \V{u}){\color{red} G}}{\V{u} \otimes \nabla \eta} + \iptr{(\V{u} \otimes \nabla \eta) G}{D(\eta \V{u})}  \\ & \ \ \ - \iptr{(\V{u} \otimes \nabla \eta ) G}{\V{u} \otimes \nabla \eta} {\color{red} \big|}^{\frac{p-2}{2}} \iptr{D(\eta \V{u}) {\color{red} G}  }{D(\eta \V{u})}\, dx \end{align*}

so that \begin{equation*} {\color{red} \N(\V{u}, \eta)} \leq \sum_{j = 1}^7 I_j \end{equation*} where

 \begin{align*} &  I_1  := \left|\int_{\Omega}  \A(x, \V{u}, \eta) [\B(x, \V{u}, \eta)]^{\frac{p-2}{2}} \, dx \right| = \left|\int_{\Omega} \iptr{F}{D(\eta^p \V{u})} \, dx \right|\\ &  I_2 := \int_\Omega {\color{red} [}\B(x, \V{u}, \eta){\color{red} ]}^\frac{p-2}{2}  |\langle(\V{u} \otimes \nabla \eta )G {\color{red} , } D(\eta \V{u})\rangle_{\text{tr}}| \, dx
 \\ & I_3  :=  \int_\Omega {\color{red} [}\B(x, \V{u}, \eta){\color{red}]}^{\frac{p-2}{2}} |\iptr{D(\eta \V{u}) G}{\V{u} \otimes \nabla \eta}|\,dx
 \\ & I_4 := \int_\Omega {\color{red}[}\B(x, \V{u}, \eta){\color{red}]}^{\frac{p-2}{2}}  |\iptr{(\V{u} \otimes \nabla \eta )G}{\V{u} \otimes \nabla \eta}| \, dx
 \\ & I_5 := \int_\Omega |{\color{red}\iptr{D(\eta \V{u}) G}{\V{u} \otimes \nabla \eta}}|^\frac{p-2}{2} {\color{red} \cancel{|}}\iptr{D(\eta \V{u}) G}{D(\eta \V{u})}{\color{red}\cancel{|}} \, dx
 \\ & I_6 := \int_\Omega  |{\color{red}\iptr{(\V{u} \otimes \nabla \eta)G}{D(\eta \V{u})}}|^\frac{p-2}{2} {\color{red}\cancel{|}}\iptr{D(\eta \V{u}) G}{D(\eta \V{u})}{\color{red}\cancel{|}} \, dx
 \\ & I_7 :=  \int_\Omega  |\iptr{(\V{u} \otimes \nabla \eta ) G}{\V{u} \otimes \nabla \eta}|^\frac{p-2}{2} {\color{red}\cancel{|}}\iptr{D(\eta \V{u}) G}{D(\eta \V{u})}{\color{red}\cancel{|}} \, dx. \end{align*}  We finish the proof by bounding each of these terms.  Let $\epsilon > 0$.  First, we have

 \begin{align*} I_1 & \leq \int_{\Omega} |\iptr{F}{D(\eta^p \V{u})}| \, dx
 \\ & \leq (p - 1)\int_{\Omega} |\eta| ^{p-2} |\iptr{F}{{\color{red}(\eta \V{u})} \otimes \nabla \eta}| \, dx + \int_\Omega  |\eta|^{p - 1} |\iptr{F}{D(\eta \V{u}) }| \, dx
 \\ & \leq  \int_{\Omega}  |\nabla \eta| \|W^{-\frac{1}{p}} F\| | W^{\frac{1}{p}} \V{u}|  \, dx + \int_\Omega  \|W^{-\frac{1}{p}} F\| {\color{red}\|W^\frac{1}{p}} D(\eta \V{u}) {\color{red}\|} \, dx
 \\ & \leq  \left(\int_{B_r} \|W^{-\frac{1}{p}} F\|^{p'} \, dx \right)^\frac{1}{p'} \left(\int_{B_r}   |\nabla \eta|^p   | W^{\frac{1}{p}} \V{u}|^p   \, dx\right)^\frac{1}{p}
 \\ & \ \ \ + \left(\int_{B_r}  \|W^{-\frac{1}{p}} F\|^{p'} \, dx \right)^\frac{1}{p'} \left(\int_{B_r} {\color{red}\| W^{\frac{1}{p}}} D(\eta \V{u}) {\color{red}
\|} ^p \,  dx\right)^\frac{1}{p}
 \\ & \lesssim C(\epsilon) \int_{{B_r}} \|W^{-\frac{1}{p}} F\|^{p'} \, dx + \frac{1}{r^p} \int_{B_r \backslash B_{r/2}} |W^\frac{1}{p} \V{u}| ^p \, dx  + \epsilon \int_{B_r} \|W^\frac{1}{p} D(\eta \V{u})\|^p \, dx \end{align*} by H\"{o}lder's inequality and Young's inequality with $\epsilon$.

 Next we estimate  $| \B(x, \V{u}, \eta)|$.  Note that by (ii') we immediately get \begin{align*} | \B(x, \V{u}, \eta)|  & \leq \|W^\frac{1}{p}  D(\eta \V{u}) \|^2 + 2 \|W^\frac{1}{p} D(\eta \V{u}) \| \|W^\frac{1}{p}  (\V{u} \otimes \nabla \eta)\| + \|W^\frac{1}{p}  (\V{u} \otimes \nabla \eta)\|^2  \\ & \leq \|W^\frac{1}{p}  D(\eta \V{u}) \|^2 + 2|\nabla \eta| \|W^\frac{1}{p} D(\eta \V{u}) \| |W^\frac{1}{p}  \V{u} | + |\nabla \eta| ^2|W^\frac{1}{p}  \V{u} |^2 \\ & \lesssim \|W^\frac{1}{p}  D(\eta \V{u}) \|^2 + |\nabla \eta| ^2|W^\frac{1}{p}  \V{u} |^2 \end{align*} so \begin{equation}  {\color{red}[}\B(x, \V{u}, \eta){\color{red}]}^{\frac{p-2}{2}} \lesssim \|W^\frac{1}{p}  D(\eta \V{u}) \|^{p - 2} + |\nabla \eta| ^{p - 2} |W^\frac{1}{p}  \V{u} |^{p - 2}.  \label{BEst} \end{equation}

Thus, by (ii') and  \eqref{BEst} we have \begin{align*} I_2 & \leq \int_\Omega |\nabla \eta| ^{p - 1} |W^\frac{1}{p}  \V{u}|^{p - 1} \|W^\frac{1}{p}  D(\eta \V{u})\| \, dx + \int_\Omega |\nabla \eta| |W^\frac{1}{p} \V{u}| \|W^\frac{1}{p}  D(\eta \V{u})\|^{p-1} \, dx \\ & \leq \left(\int_\Omega |\nabla \eta|^p |W^\frac{1}{p} \V{u}|^p \, dx \right)^\frac{1}{p'} \left(\int_\Omega \|W^\frac{1}{p}  D(\eta \V{u})\|^p \, dx \right)^\frac{1}{p} \\ & \ \ \ + \left(\int_\Omega |\nabla \eta|^p |W^\frac{1}{p} \V{u}|^p \, dx \right)^\frac{1}{p} \left(\int_\Omega \|W^\frac{1}{p}  D(\eta \V{u})\|^p \, dx \right)^\frac{1}{p'} \\ & \leq\epsilon \int_{B_r} \|W^\frac{1}{p}  D(\eta \V{u})\|^p \, dx + \frac{C(\epsilon)}{r^p} \int_{{B_r \backslash B_{r/2} }} |W^\frac{1}{p}  \V{u} |^p \, dx. \end{align*} For some constant $C(\epsilon)$ depending on $\epsilon$.  By the symmetry of (ii') we have that $I_3$ satisfies the same condition.

Similarly, using H\"{o}lder's inequality with respect to $p/2 $ we have \begin{align*} I_4 &\leq \int_\Omega |\nabla \eta |^p |W^\frac{1}{p}  \V{u}|^p \, dx + \int_\Omega |\nabla \eta|^2 \|W^\frac{1}{p}  D(\eta \V{u}) \|^{p - 2} |W^\frac{1}{p}  \V{u}|^2 \, dx
 \\ & \lesssim  \frac{1}{r^p} \int_{B_r \backslash B_{r/2} } |W^\frac{1}{p}  \V{u}|^p \, dx + \left(\frac{1}{r^p} \int_{B_r \backslash B_{r/2} } |W^\frac{1}{p}  \V{u}|^p \, dx \right)^\frac{2}{p} \left(\int_{B_r} {\color{red}\|}W^\frac{1}{p}  D(\eta \V{u}) {\color{red}\|}^p \, dx \right)^\frac{p-2}{p}
 \\ & \leq \frac{C(\epsilon)}{r^p} \int_{B_r \backslash B_{r/2} } |W^\frac{1}{p}  \V{u}|^p \, dx + \epsilon \int_{B_r} {\color{red}\|}W^\frac{1}{p}  D(\eta \V{u}){\color{red}\|}^p \, dx. \end{align*}

Likewise, H\"{o}lder's inequality with respect to $2p/(p+2)$ gives \begin{align*} I_5 & \leq \int_\Omega |\nabla \eta|^{\frac{p-2}{2}} \|W^\frac{1}{p}  D(\eta \V{u})\|^\frac{p + 2}{2} |W^\frac{1}{p}  \V{u}|^\frac{p-2}{2} \ dx \\ & \leq \left(\int_\Omega |\nabla \eta |^p |W^\frac{1}{p}  \V{u} |^p \, dx \right)^{\color{red}\frac{p - 2}{2p}} \left(\int_\Omega \|W^\frac{1}{p}  D(\eta \V{u})\|^p \, dx \right)^{\color{red}\frac{p + 2}{2p}} \\ & \lesssim \epsilon \int_{B_r} \|W^\frac{1}{p} D(\eta \V{u})\|^p \, dx + \frac{C(\epsilon)}{r^p} \int_{B_r \backslash B_{r/2} } |W^\frac{1}{p}  \V{u}|^p \, dx \end{align*}  and note that $I_6$ is estimated in exactly the same way.

Finally, \begin{align*} I_7 & \leq \int_\Omega |\nabla \eta | ^{p - 2} |W^\frac{1}{p}  \V{u}|^{p - 2} \|W^\frac{1}{p}  D(\nabla \V{u})\|^2 \, dx \\ & \leq \left(\int_\Omega |\nabla \eta|^p |W^\frac{1}{p}  \V{u}|^p \, dx \right)^\frac{p-2}{p} \left(\|W^\frac{1}{p}  D(\eta \V{u})\|^p \, dx \right)^\frac{2}{p} \\ & \leq \epsilon \int_{B_r} \|W^\frac{1}{p}  D(\eta \V{u})\|^p \, dx + \frac{C(\epsilon)}{r^p} \int_{B_r \backslash B_{r/2} } |W^\frac{1}{p}  \V{u}|^p \, dx. \end{align*}

Combining everything and setting $\epsilon$ small enough we have  \begin{align*} \int_{B_{r/2}} \|W^\frac{1}{p} D({\color{red} \cancel{\eta}} \V{u}) \|^p \, dx & \leq  \int_{B_r} \|W^\frac{1}{p} D(\eta \V{u}) \|^p \, dx \\  &  \lesssim \int_{B_r} \|W^{-\frac{1}{p}} F \|^{p'} \, dx  + \frac{1}{r^p} \int_{B_r \backslash B_{r/2}} |W^\frac{1}{p} \V{u} |^p \, dx  \end{align*} \noindent {\color{red} since $\eta \equiv 1 $ on $B_{r/2}.$}

\end{proof}

We can now prove Theorem Theorem \ref{FBTThm}.

\noindent \textit{Proof of Theorem \ref{FBTThm}}: The proof is very similar to the proof of Theorem \ref{revmeyer}. Let $\epsilon > 0$ be chosen where Theorem \ref{MainThmPoin} is true, so by \eqref{NLCaccio} \begin{align*} & \left( \frac{1}{|B_{r/2}| } \int_{B_{r/2}} \|W^{\frac{1}{p}}  D \V{u}  \|^p \, dx\right)^\frac{1}{p}   \\ & \lesssim \left(\frac{1}{|B_r|} \int_{B_r} \|W^{-\frac{1}{p}}  F \|^{p'} \, dx \right)^\frac{1}{p}     + \frac{1}{r} \left(\frac{1}{|B_r|} \int_{B_r} |W^{\frac{1}{p}}  (\V{u}  - \V{u}_{B_r})|^p \, dx \right)^\frac{1}{p}  \\ & \lesssim \left(\frac{1}{|B_r|} \int_{B_r} \|W^{-\frac{1}{p}}  F \|^{p'} \, dx \right)^\frac{1}{p}  + \left(\frac{1}{|B_r|} \int_{B_r} \|W^\frac{1}{p}  D\V{u} \|^{p - \epsilon} \, dx \right)^\frac{1}{p - \epsilon}.\end{align*}  However, setting \begin{equation*} U(x) = \|W^\frac{1}{p} (x) D\V{u} (x)\|^{p - \epsilon}, \ \ \  G(x) = \|W^{-\frac{1}{p}} (x) F(x)\|^{\frac{p'(p - \epsilon)}{p}}, \text{ and } s = \frac{p}{p - \epsilon}  \end{equation*}  we have that  \begin{equation*} \frac{1}{|B_{r/2}|}  \int_{B_{r/2}} (U(x))^s \, dx  \lesssim \left( \frac{1}{|B_{r}|} \int_{B_{r}} U(x) \, dx \right)^s + \frac{1}{|B_r|} \int_{B_r} (G(x))^s \, dx . \end{equation*}  Again Lemma $2.2$ in \cite{FBT} now says that there exists $t > s = \frac{p}{p - \epsilon}$ where \begin{align*}  & \left( \frac{1}{|B_{r/2}|} \int_{B_{r/2}} \|W^\frac{1}{p} D\V{u}  \|^{t(p - \epsilon)} \, dx \right)^\frac{1}{t} \\ & \lesssim \left( \frac{1}{|B_{r}|} \int_{B_{r}} {\color{red} \|} W^\frac{1}{p}  D\V{u} {\color{red}\|}^p \, dx \right)^\frac{p- \epsilon}{p}  + \left(\frac{1}{|B_r|} \int_{B_r} \|W^{-\frac{1}{p}}  F\|^{\frac{tp' (p - \epsilon)}{p}} \, dx \right)^\frac{1}{t}. \end{align*} Setting $q = t(p - \epsilon) > p$ clearly completes the proof. \hfill $\square$

Finally, note that (thanks to \eqref{NLCaccio}) the same arguments used to prove Theorem \ref{thm:locreg} also prove the following

{\color{red}
\begin{theorem} \label{nonLinReg} Let $p \geq d$ and let $W$ and $G$ satisfy \begin{itemize} \item $(i')  \ip{\eta G(x)}{\eta}_{\tr}   \gtrsim \|W^{1/p}(x)\eta\|^p,  \qquad \eta \in \Mnd$
\item $(ii') |\ip{\eta G(x)}{\nu}_{\tr} | \lesssim  \|W^{1/p}(x)\eta\|^{p-1}\|W^{1/p} (x)\nu\|, \qquad \eta,\nu\in \Mnd$ and $F \in L^{p'}(\Omega, W^{-\frac{p'}{p}})$. \end{itemize}

    Assume $\V{u}$ is a weak solution to \begin{equation*}  \Div \left[ \ip{D\V{u} G}{D\V{u}}_{\tr} ^{\frac{p-2}{2}} D\vec{u} G \right] = 0 \end{equation*}  Suppose that $B_{7R} \subseteq \Omega$ is an open ball of radius $7R$ and $B = B_R$ is the concentric ball with radius $R$. Then there exists $\epsilon > 0$ depending on $[W]_{\text{A}_p}$ such that for $x, y \in B$, we have both of the following:    \begin{itemize}
  \item $A) \ $  \begin{equation*} |\V{u}(x) - \V{u} (y)| \lesssim C_{x, y} R^{-d\pr{2-\frac{1}{p}}\epsilon} |x-y|^{\frac{1}{d} - \frac{1}{p} + \epsilon} \end{equation*} where  \begin{equation*} C_{x, y} = \left(\sup \frac{1}{|B'|^{1 - \epsilon}} \int_{B'} \|W^{-\frac{1}{p}}(\xi)  \|^{p'}  \, d\xi\right)^\frac{1}{p'} \end{equation*} where the supremum is over balls $B' \subset \Omega$ centered either at $x$ or $y$, and having radius $\leq 2|x - y|$.

 \item $B) \ $ \begin{equation*} |\V{u}(x) - \V{u} (y)| \lesssim \tilde{C}_{x, y} R^{-d\pr{2-\frac{1}{p}}\epsilon} |x - y|^\epsilon \end{equation*} where  \begin{equation*} \tilde{C}_{x, y} = \left(\sup \frac{1}{|B'|^{1 - \epsilon - \frac{p-d}{d(p-1)}}} \int_{B'} \|W^{-\frac{1}{p}}(\xi)  \|^{p'}  \, d\xi\right)^\frac{1}{p'} \end{equation*} where again the supremum is over balls $B' \subset \Omega$ centered either at $x$ or $y$, and having radius $\leq 2|x - y|$. \end{itemize} \end{theorem} }

\noindent {\color{red} \sout{Note that as before $\epsilon$ can be taken to be a constant multiple of $\Ap{W}$ to a constant power.}}

We will end this paper with the remark that Lemma \ref{FBTLem} most likely holds for the more general elliptic systems considered in \cite{FBT} (but with a matrix A${}_p$ degeneracy. In particular, Theorem \ref{FBTThm}  and Theorem {\color{red}\ref{nonLinReg}}  most likely holds for the system  \begin{equation*} \Div \left[ \ip{G D\V{u} }{D\V{u}}_{\tr} ^{\frac{p-2}{2}} G D\vec{u}  \right] = - \Div F \end{equation*} where $G : \Omega \rightarrow \Mn$ (and where $F = 0$ for Theorem \ref{nonLinReg})  with \vspace{3mm}

(i') $ \ip{G(x) \eta }{\eta}_{\tr}   \gtrsim \|W^{1/p}(x)\eta\|^2,  \qquad \eta \in \Mnd$,

(ii') $|\ip{ G(x) \eta}{\nu}_{\tr} | \lesssim \|W^{1/p}(x)\eta\|\|W^{1/p} (x)\nu\|, \qquad \eta,\nu\in \Mnd$ and $F \in L^{p'}(\Omega, W^{-\frac{p'}{p}})$.

\end{document}